\documentclass[12pt, reqno]{amsart}

\usepackage{textcomp} 
\usepackage{amsmath, amsfonts,amsthm,amssymb,amsxtra}
\usepackage{amscd}
\usepackage{enumitem}

\usepackage{tikz-cd} 

\usepackage{xcolor}
\usepackage{graphicx}
\usepackage{verbatim}
\newtheorem{theorem}{Theorem}[section]
\newtheorem{proposition}[theorem]{Proposition}
\newtheorem{lemma}[theorem]{Lemma}
\newtheorem{corollary}[theorem]{Corollary}

\theoremstyle{definition}

\newtheorem{definition}[theorem]{Definition}
\newtheorem{examples}[theorem]{Example}

\theoremstyle{remark}

\newtheorem{remark}[theorem]{Remark}

\numberwithin{equation}{section}

\renewcommand{\epsilon}{\varepsilon}

\renewcommand{\phi}{\varphi}
\newcommand{\R}{\mathbb{R}}

\newcommand{\Tw}{\mathrm{Tw}\,}

\begin{document}

\title{Prime-localized Weinstein subdomains}
	
\author{Oleg Lazarev and Zachary Sylvan}

\begin{abstract}
    For any high-dimensional Weinstein domain and finite collection of primes, we construct a Weinstein subdomain whose wrapped Fukaya category is a localization of the original wrapped Fukaya category away from the given primes. When the original domain is a cotangent bundle, these subdomains form a decreasing lattice whose order cannot be reversed.  
    		
    Furthermore, we classify the possible wrapped Fukaya categories of Weinstein subdomains of a cotangent bundle of a simply connected, spin manifold, showing that they all coincide with one of these prime localizations. In the process, we describe which twisted complexes in the wrapped Fukaya category of a cotangent bundle of a sphere are isomorphic to genuine Lagrangians.
\end{abstract}

\maketitle

\section{Introduction}\label{sec: intro}
\thispagestyle{empty}

\subsection{Main results}\label{subsection: main_results}

One of the main problems in symplectic topology is to understand the set of Lagrangians $L$ in a symplectic manifold $X$. For example, Arnold's nearby Lagrangian conjecture states that  any  closed exact Lagrangians $L$ in $T^*M^n_{std}$ is Hamiltonian isotopic to the zero-section $M \subset T^*M_{std}$; by work \cite{FukSS, Abouzaid, kragh_parametrized_ring_spectra_NLC}
on this conjecture, all such Lagrangians are homotopy equivalent to $M^n$. Each closed exact Lagrangian $L \subset X$ gives a Liouville subdomain $T^*L$ of $X$ and the \textit{skeleton} of $T^*L$, the stable set of its Liouville vector field, is precisely $L$. More generally, any \textit{Weinstein} domain $V$ deformation retracts to a possibly singular Lagrangian skeleton. Therefore a  Weinstein \textit{subdomain} $V \subset X$ can be considered a singular Lagrangian in $X$.
In this paper, we consider the problem of constructing and classifying Weinstein subdomains of a fixed Weinstein domain, as well as the \textit{wrapped Fukaya categories} $\mathcal W(V; R)$ of such subdomains (here, $R$ is a commutative coefficient ring). We will only consider Weinstein subdomains $V \subset X$ with the stronger property that $X \backslash V$ is also a Weinstein cobordism, i.e. $V$ is the sublevel set of an ambient Weinstein Morse function on $X$; see \cite{CE12} for background on the geometry of Weinstein domains.
	
There is a (cohomologically) fully faithful embedding of $\mathcal{W}(X; R)$ into $\Tw\mathcal{W}(X; R)$, the category of \textit{twisted complexes} on $\mathcal{W}(X; R)$. Since $\Tw\mathcal{W}(X; R)$ is a formal algebraic enlargement of a geometric category, this functor is usually not a  quasi-equivalence. To understand which $A_\infty$-categories \emph{actually arise} from Weinstein subdomains, it turns out we will have to understand which twisted complexes come from actual geometric Lagrangians. In other words, we will largely be concerned with understanding the image of this embedding. We give examples when this functor is a quasi-equivalence (Proposition \ref{prop: disks_in_cotangent_disk}) and describe its image when $X = T^*S^n_{std}$ (Example \ref{ex: cotangent_sphere}); see Section \ref{subsection: outline}. This type of question about the geometricity of twisted complexes has previously been studied by \cite{HKK, Auroux_Smith_surfaces}.

Given a small $A_\infty$-category $\mathcal{C}$ over $\mathbb{Z}$ and set of objects $\mathcal{A}$ of $\mathcal{C}$, one can form the quotient $A_\infty$-category $\mathcal{C}/\mathcal{A}$, which comes with a localization functor $\mathcal{C} \rightarrow \mathcal{C}/\mathcal{A}$; see \cite{Lyubashenko--Manzyuk, Lyubashenko--Ovsienko}. In particular, given a collection of prime numbers $P \subset \mathbb{Z}$, one can form 	
\begin{equation}
    \mathcal{C}\left[\frac{1}{P}\right]:=\mathcal{C}/\left\{\mathrm{cone}(p\cdot\mathrm{Id}_L)\,\vert\, p\in P,L\in\mathcal{C}\right\}
\end{equation}
the localization of $\mathcal{C}$ away from the primes $P$.
Quotienting by $\mathrm{cone}(p \cdot Id_L)$ kills the object $\mathrm{cone}(p \cdot Id_L)$, which has the effect of making the morphism $p \cdot Id_L$ a quasi-isomorphism, i.e. inverting $p$. Hence if $\hom^*_{\mathcal{C} } (L, K)$ is a cochain complex of free Abelian groups, then $\hom^*_{\mathcal{C}[1/P]}(L,K)$ is quasi-isomorphic to $\hom^*_{\mathcal{C}}(L, K) \otimes_{\mathbb{Z}} \mathbb{Z}[\frac{1}{P}]$, which explains our notation $\mathcal{C}[\frac{1}{P}]$.
We will also allow $P$ to be empty or contain $0$, in which case  $\mathcal{C}[\frac{1}{P}]$ is  the original category $\mathcal{C}$ or the  trivial category, respectively.

Our first result is that any high-dimensional Weinstein domain has Weinstein subdomains whose Fukaya categories are localizations away from any finite collection of primes $P$. Furthemore, these subdomains are \textit{almost} symplectomorphic, i.e. their symplectic forms are homotopic through non-degenerate 2-forms, and 
hence indistinguishable 
from the point of view of classical smooth topology. We note that by Gromov's h-principle \cite{gromov_hprinciple} for open symplectic manifolds, any two almost symplectomorphic  Weinstein domains are actually homotopic through symplectic structures (but may not be symplectomorphic). 

\begin{theorem}\label{thm: exotic_subdomains}
	For any Weinstein domain $X^{2n}$ with $n \ge 5$ and finite collection of prime numbers $P$, that is possibly empty or contains $0$, 
	there is a Weinstein subdomain $X_{P} \subset X$ such that
	$\Tw\mathcal{W}(X_P; \mathbb{Z}) \cong \Tw\mathcal{W}(X; \mathbb{Z})[\frac{1}{P}]$ and the Viterbo transfer functor
	\[
	\mathcal V\colon \Tw\mathcal{W}(X; \mathbb{Z})\to \Tw\mathcal{W}(X_P; \mathbb{Z})
	\]
	is localization away from $P$. In particular, $\Tw\mathcal{W}(X_P; \mathbb{F}_p) = 0$ if $p \in P$ or $0\in P$, and $ \Tw\mathcal{W}(X_P; \mathbb{F}_p) \cong  \Tw\mathcal{W}(X; \mathbb{F}_p)$ otherwise. Furthermore, we can arrange that 
	\begin{enumerate}
		\item 
		The Weinstein cobordism $X \backslash X_P$ is smoothly  trivial and hence  $X_P$ is almost symplectomorphic to $X$. \\
		\item 
		If $Q \subset P$ or $0 \in P$, we can exhibit a Weinstein embedding
		$\phi_{P,Q}: X_P \hookrightarrow X_Q$ with the property that if $R \subset Q \subset P$, then $\phi_{P,Q} \circ \phi_{Q,R}$ is Weinstein homotopic to $\phi_{Q,R}$.
		\item If $P$ is empty, then $X_P$ is $X$. If $0 \in P$, then $X_P$ is the flexibilization $X_{flex}$ of $X$ and the Weinstein embedding $X_{P} \subset X$ is unique up to Weinstein homotopy. 
	\end{enumerate}
\end{theorem}
	
\begin{remark}
	For us, the objects of $\mathcal W(X; R)$ are graded exact spin Lagrangian submanifolds (branes) in $X$ that are closed or have conical Legendrian boundary in a collar of $\partial X$. We will usually not specify what type of grading data our Lagrangian should have, except when $X$ is a cotangent bundle and we will use the canonical $\mathbb{Z}$-grading.
	
	In Section \ref{sec:C*_mod} we will briefly allow some branes to be equipped with rank $1$ local systems. We will generally not treat these as honest members of $\mathcal W(X; R)$, but they are certainly (isomorphic to) members of $\Tw\mathcal W(X; R)$.
\end{remark}

More precisely, there is a Weinstein homotopy of the Weinstein structure on $X$ to a different structure $X'$ so that $X_P$ is a sublevel set of the Weinstein Morse function on $X'$. 
That is, $X_P$ is itself a Weinstein domain \textit{and} $X' \backslash X_P$ is a Weinstein cobordism. 
We also note that Theorem \ref{thm: exotic_subdomains} holds for any grading of $X$ (and the induced grading on its subdomains). 

Our construction is related to a result of Abouzaid and Seidel \cite{abouzaid_seidel_recombination}, who also showed that any Weinstein domain $X^{2n}, n \ge 6,$ can be modified to a produce a new Weinstein domain $X_P'$, almost symplectomorphic to $X$, with the property that $SH^*(X_P; \mathbb{F}_q) \cong SH^*(X; \mathbb{F}_q)$ if $q \not \in P$ and $SH(X_P; \mathbb{F}_q) = 0$ otherwise. Theorem \ref{thm: exotic_subdomains} proves this property on the level of Fukaya categories, which implies the result on the level of symplectic cohomology \cite{Ganatra_thesis}. 
The other main difference between our domain $X_P$ and the domain $X_P'$ produced by Abouzaid and Seidel \cite{abouzaid_seidel_recombination} is that $X_P$ is manifestly a subdomain of $X$ while $X_P'$ is an abstract Weinstein domain. The construction of Abouzaid and Seidel involves modifying a Lefschetz fibration for $X$ by enlarging the fiber and adding new vanishing cycles, and there is no obvious map between $X$ and $X_P'$. Our construction involves removing a certain regular Lagrangian disk (which also appears in Abouzaid-Seidel's work) so that $X_P$ is automatically a subdomain of $X$; constructing these regular disks requires $n \ge 5$, hence the restriction on $n$ in Theorem \ref{thm: exotic_subdomains}. 
Both our construction and that of Abouzaid-Seidel require many choices, but we conjecture that one can make these choices so that the resulting Weinstein domains $X_P, X_P'$ agree.

\begin{remark}
	We expect that Theorem \ref{thm: exotic_subdomains} also holds for an infinite collection of primes $P$ if we allow $X_P$ to be a symplectic manifold that is the intersection of infinitely many Weinstein domains. Namely, if $P = \{p_1, p_2, \dotsc\}$ and $P_i :=\{p_1, \dotsc, p_i\}$, then Theorem \ref{thm: exotic_subdomains} provides a decreasing collection of Weinstein subdomains $X \supset X_{P_1} \supset X_{P_2} \supset X_{P_3} \supset \dotsb $ and we can set $X_P := \bigcap_{i\ge 1} X_{P_i}$. Then we expect that $\Tw\mathcal{W}(X_P) \cong \Tw\mathcal{W}(X)[\frac{1}{P}]$ (or can take this as a definition).
	However, we do not know whether $X_P$ is a Weinstein \emph{manifold} in the sense of \cite{CE12}, i.e.  an increasing \emph{union} of finite type Weinstein domains. If $P$ is the set of all primes, we consider $X_P$ to be a symplectic `rationalization' of $X$, analogous to the rationalization of classical spaces.
\end{remark}

\begin{remark}\label{rem: stops}
	An analog of Theorem \ref{thm: exotic_subdomains} is true for Weinstein domains with Weinstein stops. 
	For example, 
	in Theorem \ref{thm: p-handles} we prove that there is a Legendrian sphere $\Lambda_P \subset \partial B^{2n}_{std}$ so that 
	$\Tw\mathcal{W}(B^{2n}_{std}, \Lambda_P) \cong \Tw\mathcal{W}(B^{2n}_{std}, \Lambda_\varnothing)[\frac{1}{P}] \cong \Tw  \mathbb{Z}[\frac{1}{P}]$, where $\Lambda_\varnothing$ is the Legendrian unknot, 
	and there is a smoothly trivial Lagrangian cobordism $L \subset \partial B^{2n}_{std} \times [0,1]$
	whose positive, negative ends $\partial_{\pm} L$ coincide with
	$\Lambda_\varnothing,\Lambda_P$ respectively.  Note that $(B^{2n}_{std}, \Lambda_\varnothing)$ is the standard Weinstein handle of index $n$; 
	we call $(B^{2n}_{std}, \Lambda_P)$ a Weinstein $P$-handle of index $n$.
	The construction of the Weinstein subdomain  $X_P$ in Theorem \ref{thm: exotic_subdomains} can be viewed as replacing all standard Weinstein handles of index $n$ with Weinstein $P$-handles. This is similar to the classical rationalization of a CW complex, in which all standard cells are replaced with `rational' cells.
\end{remark}

Next we consider Weinstein subdomains of the cotangent bundle $T^*M_{std}$ of a smooth manifold $M$. 
Using Theorem \ref{thm: exotic_subdomains} and the additional  fact that 
$\Tw\mathcal{W}(T^*M_{std}; \mathbb{F}_p)$ 
is non-trivial for any $p$, we show that $T^*M_{std}$ has many infinitely different Weinstein subdomains.

\begin{corollary}\label{cor: exotic_subdomains_cotangent}
	If $n \ge 5$,  then for any finite collection $P$ of primes numbers,  possibly empty or containing zero,  there is a Weinstein subdomain $T^*M_{P}^n \subset T^*M_{std}^n$ almost symplectomorphic to $T^*M_{std}$ so that $\Tw\mathcal{W}(T^*M_P; \mathbb{Z}) \cong  \Tw\mathcal{W}(T^*M; \mathbb{Z})[\frac{1}{P}]$. Furthermore, we can arrange for $T^*M_P$ to be a Weinstein subdomain of $T^*M_{Q}$ if and only if $Q \subset P$ or $0 \in P$, i.e. the product of primes in $P$ divides the product of those in $Q$.
\end{corollary}

The claim here is stronger than in Theorem \ref{thm: exotic_subdomains}: here  $T^*M_P$ is a Weinstein subdomain of $T^*M_{Q}$ if and \emph{only if} $Q  \subset P$, or $0 \in P$
(in fact, `Weinstein' subdomain can be replaced with `Liouville' subdomain). The proof of Corollary \ref{cor: exotic_subdomains_cotangent} carries over to any Weinstein domain $X$ for which $\mathcal{W}(X; \mathbb{F}_p)$ is non-trivial for all $p$, e.g. if $X$ has a closed exact Lagrangian. 
Furthermore, by the `only if' part of the claim, our subdomains form a decreasing lattice whose order cannot be reversed. For example, there is an infinite decreasing sequence
\[
    T^*M_{std} \supsetneq T^*M_{2} \supsetneq T^*M_{2,3} \supsetneq T^*M_{2,3,5} \supsetneq \dotsb \supsetneq T^*M_{P_k} \supsetneq \dotsb
    \supsetneq T^*M_{0} = T^*M_{flex}
\]
where $P_k$ is the set of the first $k$ primes; the other subdomains $T^*M_P$ where $P \ne P_k$, e.g. $T^*M_{7,13}$, 
contain $T^*M_{P_k}$ for sufficiently large $k$. In particular, $T^*M_{std}$ has many \emph{singular} Lagrangians given by the skeleta of $T^*M_P$. These skeleta are not Hamiltonian isotopic since otherwise we could find a Liouville embedding of $T^*M_Q$ into $T^*M_P$ for $P \supset Q$. 
We contrast this with the nearby Lagrangian conjecture which claims that all closed exact \textit{smooth} Lagrangians of $T^*M_{std}$  are Hamiltonian isotopic. Finally, we  note that $T^*M_P$ has no closed exact smooth Lagrangians if $P$ is non-empty since its Fukaya category over $\mathbb F_p$ vanishes.
	
Our second main result about subdomains of $T^*M_{std}$ is a converse to Corollary \ref{cor: exotic_subdomains_cotangent}: the Fukaya category of \textit{any} Weinstein subdomain of $T^*M_{std}$ is a localization of $\Tw\mathcal{W}(T^*M_{std}; \mathbb{Z})$ away from some finite collection of primes. Here we use the $\mathbb{Z}$-grading on  $T^*M_{std}$ and its subdomains induced by the Lagrangian fibration by cotangent fibers.
\begin{theorem}\label{thm:Wein_subdom_classification}
	If  $M^n$ is a closed, simply connected, spin manifold and $i\colon X \hookrightarrow T^*M_{std}$ is a Weinstein subdomain, then 
	$\Tw\mathcal{W}(X; \mathbb{Z}) \cong  \Tw\mathcal{W}
	(T^*M_{std}; \mathbb{Z}) [\frac{1}{P}]$ 	
	for some finite collection of primes $P$, that is possibly empty or contains $0$, and is unique (unless $P$ contains $0$). Under this equivalence, the Viterbo transfer functor
		$\Tw\mathcal{W}(T^*M; \mathbb{Z}) \rightarrow \Tw\mathcal{W}(X; \mathbb{Z})$ is localization away from $P$. 
	Furthermore, either the restriction $i^*\colon H^n(T^*M^n; \mathbb{Z}) \rightarrow H^n(X; \mathbb{Z})$ is an isomorphism or $\mathcal{W}(X; \mathbb{Z}) \cong 0$ (or both).
\end{theorem}
	
For $n \ge 5$, Theorem \ref{thm:Wein_subdom_classification} combined with Corollary \ref{cor: exotic_subdomains_cotangent} completely classify which categories appear as Fukaya categories (with integer coefficients) of Weinstein subdomains of cotangent bundles of closed, simply connected, spin manifolds. For $n \le 4$, the question remains open whether the categories $\Tw\mathcal{W}(T^*M_{std}^n; \mathbb{Z})[\frac{1}{P}]$ actually appear as Fukaya categories of subdomains. Indeed, in the $n = 1$ case, the only subdomains of $T^*S^1_{std}$ are $T^*S^1_{std}$ or $B^{2}_{std}$, which algebraically correspond to the cases $P = \varnothing$ and $P = 0$. We note that the condition on the map $i^*$ shows that any Weinstein ball $\Sigma \subset T^*M_{std}$ has trivial $\mathcal{W}(\Sigma)$. There are no restrictions on $i^*$ in degrees \textit{less than} $n$, as in the case of $T^*M_{std} \cup H^{n-1} \subset T^*M_{std}$. Finally, we note that the `both' case does occur in the case of $T^*M^n_{flex} \subset T^*M^n_{std}$.

We emphasize that Theorem \ref{thm:Wein_subdom_classification} classifies \textit{Weinstein} subdomains of $X \subset T^*M_{std}$; namely, $X$ is itself a Weinstein domain \textit{and} $T^*M_{std}\backslash X$ is a Weinstein cobordism (after Weinstein homotopy of $T^*M_{std}$). We do not know if our result holds for more general Liouville subdomains $X \subset T^*M_{std}$, for which either $X$ is not a Weinstein domain or  $T^*M_{std}\backslash X$ is not a Weinstein cobordism. However, in the only known examples of subdomains $X \subset T^*M_{std}$ for which  $T^*M_{std}\backslash X$ is not a Weinstein cobordism, $X$ is a flexible domain \cite{EM} and hence has trivial Fukaya category. Furthermore, our classification is quite special to cotangent bundles: for a general Weinstein domain $X$, there are subdomains $X_0$ for which $\Tw\mathcal{W}(X_0)$ is different from $\Tw\mathcal{W}(X)[\frac{1}{P}]$ for any collection of primes $P$. For example,  the boundary connected sum $T^*M_{std} \natural T^*N_{std}$ of two cotangent bundles $T^*M, T^*N$ has a natural collection subdomains indexed by \textit{pairs} of collections of primes $P, Q$, namely $T^*M_P\natural T^*N_Q$.

\subsection{Outline of proofs}\label{subsection: outline}
	
We now outline the proofs of our two main results: Theorem \ref{thm: exotic_subdomains} and Theorem \ref{thm:Wein_subdom_classification}. We focus primarily on the latter result, whose proof involves describing which twisted complexes in $\Tw \mathcal{W}(T^*M_{std})$ are quasi-isomorphic to actual Lagrangians, i.e. the image of the functor $\mathcal{W}(T^*M_{std}) \hookrightarrow \Tw \mathcal{W}(T^*M_{std})$.

To see the connection, consider a Weinstein subdomain $X_0^{2n} \subset X^{2n}$. The Weinstein cobordism $X\setminus X_0$ has index $n$ Lagrangian co-core disks $D_1, \dotsc, D_k$ which are objects of $\mathcal{W}(X)$. Ganatra, Pardon, and Shende \cite{ganatra_generation} proved that
\[
\Tw \mathcal{W}(X_0) \cong \Tw \mathcal{W}(X)/(D_1, \dotsc, D_k)
\]
and the localization functor 
\[
\Tw \mathcal{W}(X) \rightarrow \Tw\mathcal{W}(X_0)
\]
has a geometric interpretation and is called the \textit{Viterbo transfer functor}. See \cite{Sylvan_Orlov_functor} for results when $X, X_0$ are both Weinstein but $X \backslash X_0$ is not necessarily a Weinstein cobordism. So to describe $\Tw \mathcal{W}(X_0)$,  it suffices to describe the quasi-isomorphism classes of the Lagrangian disks $D_1, \dotsc, D_k$ in $\Tw \mathcal{W}(X)$. To prove Theorem \ref{thm: exotic_subdomains}, we  construct a disjoint collection of  disks $D_1, \dotsc, D_k \subset X^{2n}, n \ge 5,$ so that $\Tw \mathcal{W}(X; \mathbb{Z})/(D_1, \dotsc, D_k) \cong  \Tw \mathcal{W}(X; \mathbb{Z})[\frac{1}{P}]$. By removing the Weinstein handles associated to these disks, we get the subdomain $X_P$ with the desired property $\Tw \mathcal{W}(X_P; \mathbb{Z}) \cong  \Tw \mathcal{W}(X; \mathbb{Z})/(D_1, \dotsc, D_k)\cong \Tw \mathcal{W}(X; \mathbb{Z})[\frac{1}{P}]$.
	
\begin{remark}\label{rem: split_closure}
	In fact, the localization $\mathcal{C}/\mathcal{A}$ by some objects $\mathcal{A} \subset \mathcal{C}$ depends only on the \textit{split-closure} of $\mathcal{A}$ in $\mathcal{C}$ 	\cite{ganatra_generation}, which is the kernel of the localization $\mathcal{C}\rightarrow \mathcal{C}/\mathcal{A}$. A subcategory $\mathcal{C}' \subset \mathcal{C}$ is split-closed if 
	for any two objects $A, B$ of
	$\mathcal{C}$ for which $A\oplus B$ is an object of $\mathcal{C}'$, then  $A, B$ are also objects of $\mathcal{C}'$. 
	More generally, there is a correspondence between localizing functors $\mathcal{C} \rightarrow \mathcal{D}$ and split-closed subcategories of $\mathcal{C}$.
\end{remark}

Since any Weinstein domain $X \subset T^*M_{std}$ has  $\Tw \mathcal{W}(X) \cong \Tw \mathcal{W}(T^*M_{std})/(D_1, \dotsc, D_k)$ for some collection of Lagrangian disks $D_1, \dotsc, D_k$ in $T^*M_{std}$, to prove Theorem \ref{thm:Wein_subdom_classification} we need to classify the objects of $\Tw \mathcal{W}(T^*M_{std})$ that are quasi-isomorphic to embedded Lagrangian disks. 
By work of Abouzaid \cite{abouzaid_cotangent_fiber}, any object of $\Tw \mathcal{W}(T^*M_{std})$ is quasi-isomorphic to a twisted complex of the cotangent fibers $T^*_q M$; after taking boundary connected sums of these cotangent fibers along isotropic arcs, we can replace this twisted complex with a single embedded Lagrangian disk  \emph{equipped with a bounding cochain}. However for Theorem \ref{thm:Wein_subdom_classification}, we need to consider Lagrangian disks without bounding cochains and as we will see in Theorem \ref{thm: Lagrangian_homology_twisted} below, not every twisted complex in 
$\Tw \mathcal{W}(T^*M_{std})$ is quasi-isomorphic to such a disk.

In the following key result, we characterize those twisted complexes in $\Tw  \mathcal{W}(T^*M_{std})$ that are quasi-isomorphic to Lagrangian disks. To make this precise, we fix some notation. Let $A$ be an object of some pre-triangulated $A_\infty$-category $\mathcal{C}$ over $\mathbb Z$. A homotopy unit $e\in \mathrm{end}_\mathcal{C}(A)$ of $A$ gives an $A_\infty$-homomorphism $\mathbb Z\rightarrow \mathrm{end}_\mathcal{C}(A)$, which induces a functor $\Tw \mathbb Z \rightarrow \Tw \mathrm{end}_\mathcal{C}(A)$. 
Applying this to $\mathcal{C} = \Tw \mathcal{W}(T^*M_{std}; \mathbb{Z})$, and $A = T^*_q M$, we get the composition of functors 
\begin{equation}\label{eqn: tensor_functor}
	\otimes T^*_q M\colon\Tw \mathbb{Z} \rightarrow 
	\Tw \mathrm{end}(T_q^*M) \xrightarrow{\sim}
	\Tw \mathcal{W}(T^*M_{std}; \mathbb{Z})  
\end{equation}
Note here that $\Tw \mathbb{Z}$ is the category of finite cochain complexes, i.e. those $\mathbb Z$-cochain complexes whose underlying graded Abelian group is free and finitely generated. The functor $\otimes T^*_q M$ sends such a twisted complex on $\mathbb{Z}$ to the corresponding twisted complex on $T^*_q M$. In particular, the differential consists entirely of morphisms that are all integer multiples of the unit. By Abouzaid's theorems \cite{Abouzaid_based_loops, abouzaid_cotangent_fiber}, the second functor is actually a quasi-equivalence, meaning that every object of $\Tw \mathcal{W}(T^*M_{std}; \mathbb{Z})$ is a twisted complex of $T^*_q M$ with differential given by \textit{arbitrary} elements of $\mathrm{end}(T^*_q M)$. As we will see, the composite functor $\otimes T_q^*M$ is not essentially surjective, but for nice $M$ every Lagrangian disk is contained in its essential image. More generally, we have the following result.
\begin{theorem}\label{thm: Lagrangian_homology_twisted}
	Let $M^n$ be a closed, simply-connected, spin manifold and let  $i: L^n \hookrightarrow T^*M_{std}^n$ be an exact Lagrangian brane. If $i: L^n \hookrightarrow T^*M_{std}^n$ is null-homotopic as a continuous map, then $L$
	is in the image of $\otimes T^*_q M$. More precisely, 
	$L$ is quasi-isomorphic to 
	\[
	CW^*(M, L; \mathbb{Z}) \otimes T^*_q M
	\]
	in $\Tw \mathcal{W}(T^*M_{std}; \mathbb{Z})$, where the cochain complex $CW^*(M, L; \mathbb{Z})$ is considered an object of $\Tw \mathbb{Z}$. 
\end{theorem}

Combining this result with the construction of the Lagrangian disks in the Theorem \ref{thm: exotic_subdomains}, we have the following description of the image of $\otimes T^*_q M$.

\begin{corollary}\label{cor: disks_cotangent_bundle_image}
	If $M^n$ is a closed, simply-connected, spin manifold and $L \subset T^*M_{std}$ is a Lagrangian disk, then $L$ is in the essential image of $\otimes T^*_q M$. If $n \ge 5$, then every object of $\Tw  \mathcal{W}(T^*M_{std})$ in the image of $\otimes T^*_q M$ is quasi-isomorphic to a Lagrangian disk.
\end{corollary}

Theorem \ref{thm: Lagrangian_homology_twisted} translates the purely topological condition that the Lagrangian is null-homotopic into the Floer-theoretic condition on its quasi-isomorphism class in the Fukaya category. The proof of Theorem \ref{thm: Lagrangian_homology_twisted} actually shows that this topological condition can be weakened to the algebraic condition that the restriction homomorphism $i^*\colon C^*(T^*M; \mathbb{Z})\rightarrow C^*(L; \mathbb{Z})$ on singular cochain algebras is homotopic as an $A_\infty$-homomorphism to a map that factors through $\mathbb{Z}$. In Proposition \ref{prop: C* modules homotopic}, we prove a generalization of Theorem \ref{thm: Lagrangian_homology_twisted}  for arbitrary Lagrangians $i: L \hookrightarrow T^*M_{std}$ that are not nessarily null-homotopic: we prove that the $CW^*(M, M)$-module $CW^*(M, L)$ is in the image of the composition
\[
\mathrm{Mod}_{C^*(L)} \xrightarrow{i^\vee} \mathrm{Mod}_{C^*(T^*M)} \cong \mathrm{Mod}_{CW^*(M,M)},
\]
where $i^\vee$ is the pullback functor on modules induced by the restriction homomorphism $i^*: C^*(T^*M) \rightarrow C^*(L)$.

The proof of Corollary \ref{cor: disks_cotangent_bundle_image} uses the Koszul duality between the wrapped Floer cochains of a cotangent fiber and those of the zero-section of a cotangent bundle, the fact that the zero-section $M$ is homotopy equivalent to the ambient manifold $T^*M_{std}$, and a certain commutativity property of the closed-open map that holds for arbitrary Liouville domains (see Proposition \ref{prop: C* modules homotopic} and Remark \ref{rem: Hochschild cohomology}). Consequently, Corollary \ref{cor: disks_cotangent_bundle_image} is quite special to cotangent bundles and analogous results do not hold for general Weinstein domains. Even if $X^{2n}$  has a single index $n$ handle with co-core $D^n$, then it is not true that any Lagrangian disk $L \subset X$ is isomorphic to $C^* \otimes D$ for some cochain complex $C^*$ over $\mathbb{Z}$ (but since $D$ is a generator of $\Tw \mathcal{W}(X)$, $L$ is isomorphic to a twisted complex of $D$ whose differential has arbitrary morphisms). For example, this is the case if $X^{2n}$ is one of the exotic cotangent bundles constructed in \cite{Lazarev_maximal} that have many closed regular Lagrangians with different topology.

In the following example, we illustrate the above results when $M = S^n$. We describe the image of the functor $\mathcal{W}(T^*S^n_{std}) \hookrightarrow \Tw \mathcal{W}(T^*S^n_{std})$ and give examples of Lagrangians that are not in image of the functor $\otimes T^*_q S^n: \Tw \mathbb{Z} \rightarrow \Tw \mathcal{W}(T^*S^n_{std})$.

\begin{examples}\label{ex: cotangent_sphere}
    We first Floer-theoretically classify all exact Lagrangian branes in $T^*S^n_{std}$. If $L \subset T^*S^n$ is closed, then it is quasi-isomorphic to the zero-section $S^n \subset T^*S^n$ by \cite{FukSS}; if $L^n \subset T^*S^n$ has non-empty boundary, then any embedding $i: L^n \hookrightarrow T^*S^n$ is automatically null-homotopic and hence in the image of $\otimes T^*_q M: \Tw \mathbb{Z} \rightarrow \Tw \mathcal{W}(T^*S^n_{std}; \mathbb{Z})$; this implies that $L$ is quasi-isomorphic to a disk if $n \ge 5$.	However, there are many exact Lagrangians $L \subset T^*S^n_{std}$ that are not \textit{homotopy-equivalent} to a disk or $n$-sphere:  any smooth $n$-manifold $L$ with non-empty boundary and trivial complexified tangent bundle has an exact Lagrangian embedding into $T^*S^n$ for $n \ge 3$; see \cite{EGL, Lazarev_reg_lag}. Using the above classification, one can check that for any Lagrangian $L \subset T^*S^n_{std}$ with non-empty boundary, the wrapped Floer cohomology $HW^*(L,L)$ is either trivial or infinite-dimensional (over some field $\mathbb{F}_p$). This implies the following new case of the Arnold chord conjecture: any Legendrian $\Lambda \subset ST^*S^n_{std}$ that bounds an exact Lagrangian brane (so graded, spin) in $T^*S^n_{std}$ has at least one Reeb chord for any contact for any contact form;  see \cite{Hutchings_Taubes_chord_conjecture, Ritter, Mohnke_chord} for existing results.
    
    Although all Lagrangians with non-empty boundary are in the image of the functor $\otimes T^*_q S^n$, we now show the zero-section $S^n \subset T^*S^n$ is not; this  is compatible with the fact that $i: S^n \hookrightarrow T^*S^n$ is not null-homotopic. Indeed, any Lagrangian $L$ that is in the image of \[\otimes T^*_q M \colon \Tw\mathbb{Z} \rightarrow \Tw\mathcal{W}(T^*S^n_{std})\]
    represents something in the image of the pullback functor $i^\vee : Mod_\mathbb{Z} \rightarrow Mod_{C^*(S)}$ 
    and so has the property that the product \[CW^*(S^n, L) \otimes CW^n(S^n, S^n) \rightarrow CW^{*+n}(S^n, L)\] must vanish on cohomology (since $CW^*(S^n, S^n) \cong C^*(S^n) \rightarrow \mathbb{Z}$ vanishes in degree $n$). 
    Since this product does not vanish for  $L = S^n$, this Lagrangian is not in the image of $\otimes T^*_q S^n$. However since $T^*_q S^n$ generates $\Tw  \mathcal{W}(T^*S^n_{std})$, the zero-section  $S^n$ is still some twisted complex of $T^*_q S^n$. It turns out that $S^n$ is quasi-isomorphic to $T^*_q S^n[n] \overset{\gamma}{\rightarrow} T^*_q S^n$, where $\gamma$ is the generator of
    $CW^1(T^*_q S^n[n], T^*_q S^n) = CW^{1-n}(T^*_q S^{n}, T^*_q S^{n}) \cong C_{n-1}(\Omega S^n) \cong \mathbb{Z}$.
    Note that $\gamma$  is not  a multiple of the unit.
    
    In all, we have shown that if $n \ge 5$, the image of the full and faithful embedding $\mathcal{W}(T^*S^n_{std}) \hookrightarrow \Tw \mathcal{W}(T^*S^n_{std}) \cong \Tw \left\{T_q^*S^n\right\}$ is quasi-isomorphic to the subcategory
    \[
    \left\{C^* \otimes T^*_q S^n| C^* \mbox{ is a cochain complex over } \mathbb{Z} \right\} \cup \left\{T^*_q S^n[n] \overset{\gamma}{\rightarrow} T^*_q S^n\right\}.
    \]
    For more general manifolds $M$, $\mathcal{W}(T^*M)$ has other objects besides the zero-section and Lagrangian disks, e.g. the surgery of the zero-section and a cotangent fiber. 
\end{examples}
	
Finally, we use Corollary \ref{cor: disks_cotangent_bundle_image} to prove Theorem \ref{thm:Wein_subdom_classification} classifying the wrapped Fukaya categories of subdomains of $T^*M_{std}$.

\begin{proof}[Proof of Theorem \ref{thm:Wein_subdom_classification}]
    Let $X^{2n} \subset T^*M_{std}$ be a Weinstein subdomain and  $C^{2n} := T^*M_{std}\backslash X^{2n}$ the complementary Weinstein cobordism. Then $C = C_{sub} \cup H^n_{1} \cup \dotsb \cup H^n_k$, where all handles of $C_{sub}$ are subcritical, i.e. have index less than $n$. The Viterbo restriction induces an equivalence $\Tw \mathcal{W}(X \cup C_{sub}; \mathbb{Z}) \cong \Tw \mathcal{W}(X; \mathbb{Z})$  on the subcritical cobordism \cite{ganatra_generation}.
    Also by \cite{ganatra_generation},
    \begin{align*}
        \Tw \mathcal{W}(X \cup C_{sub}; \mathbb{Z})
        &\cong \Tw \mathcal{W}(T^*M_{std} \setminus (D_1 \amalg \dotsb \amalg D_k); \mathbb{Z})\\
        &\cong \Tw \mathcal{W}(T^*M_{std}; \mathbb{Z})/(D_1, \dotsc, D_k),
    \end{align*}
    where $D_1, \dotsc, D_k \subset T^*M_{std}$ are the Lagrangian co-cores of $H^n_1, \dotsc, H^n_k$; recall that this quotient category depends just on the subcategory split-generated by these disks by Remark \ref{rem: split_closure}. Now by Corollary \ref{cor: disks_cotangent_bundle_image}, $D_i \cong CW^*(M, D_i) \otimes T^*_q M$ in $\Tw \mathcal{W}(T^*M; \mathbb{Z})$ where $CW^*(M, D_i)$ is considered as an object of $\Tw \mathbb{Z}$, or equivalently a cochain complex over $\mathbb{Z}$. Any cochain complex of free Abelian groups splits as a direct sum of twisted complexes of the form $\mathbb{Z}[1] \overset{m}{\rightarrow}\mathbb{Z}$ for some integer $m$ and free groups $\mathbb{Z}$ (and their shifts). If $CW^*(M, D_i)$ has a  $\mathbb{Z}$-summand, then $D_i$ split-generates and hence $\Tw \mathcal{W}(T^*M_{std}; \mathbb{Z})/(D_1, \dotsc, D_k)$ is trivial. Otherwise, let $p_1, \dotsc, p_j$ be the collection of primes dividing $m$ in the summand $\mathbb{Z}[1] \overset{m}{\rightarrow}\mathbb{Z}$. Then the split-closure of $\mathbb{Z}[1] \overset{m}{\rightarrow}\mathbb{Z}$ coincides with that of the objects $\mathbb{Z}[1] \overset{p_1}{\rightarrow}\mathbb{Z}, \dotsc,  \mathbb{Z}[1] \overset{p_j}{\rightarrow}\mathbb{Z}$. So if $P$ denotes the set of primes obtained this way over all $D_1, \dotsc, D_k$, the split-closure of $(D_1, \dotsc, D_k)$ coincides with that of $ T^*_q M[1]\overset{p}{\rightarrow} T^*_q M \cong \mathrm{cone}(p \cdot \mathrm{Id}_{T^*_q M})$ where $p \in P$. Since $T^*_q M$ generates $\Tw \mathcal{W}(T^*M_{std}; \mathbb{Z})$, the subcategory split-generated by $(D_1, \dotsc, D_k)$ coincides with that split-generated by
    \[
        \left\{ \mathrm{cone}(p\cdot \mathrm{Id}_L) \mid p \in P, L \in \Tw \mathcal{W}(T^*M_{std}; \mathbb{Z}) \right\},
    \]
    and so $\Tw \mathcal{W}(X; \mathbb{Z}) \cong \Tw \mathcal{W}(T^*M_{std}; \mathbb{Z})\left[\frac{1}{P}\right]$ as desired. Also, $P$ is unique since $\mathcal{W}(X; \mathbb{F}_q)$ vanishes if $q \in P$ and $\Tw \mathcal{W}(X; \mathbb{F}_q) \cong \Tw \mathcal{W}(T^*M_{std}; \mathbb{F}_q)$ is non-trivial if $q \not \in P$ and $0 \not \in P$. 
    
    Finally, if $i^*: H^n(T^*M; \mathbb{Z}) \rightarrow H^n(X; \mathbb{Z})$ is not an isomorphism, then $[D_i] \in H^n(T^*M; \mathbb{Z}) \cong \mathbb{Z}$ is non-zero for some $D_i$ and so the algebraic intersection number $M \cdot D_i \in \mathbb{Z}$ is non-zero. Since this intersection number is precisely the Euler characteristic $\chi(CW^*(M, D_i))$ of the Floer cochains $CW^*(M, D_i)$, the direct sum decomposition of $CW^*(M, D_i)$ discussed above must contain a free group $\mathbb{Z}$, which implies that  $\Tw \mathcal{W}(X; \mathbb{Z})$ is trivial.
\end{proof}
	
\begin{remark}
    Abouzaid observed that Corollary \ref{cor: disks_cotangent_bundle_image}, and hence Theorem \ref{thm:Wein_subdom_classification}, extends to the case where $M$ has finite fundamental group and spin universal cover. Indeed, in that case any Lagrangian disk $L \subset T^*M$ lifts to a disk $\tilde L \subset T^*\tilde M$. Applying Corollary \ref{cor: disks_cotangent_bundle_image} to $\tilde L$, we obtain an isomorphism
    \[
        \tilde L \cong K^* \otimes T_q^*\tilde M
    \]
    for some complex $K^* \in \Tw\mathbb Z$. Presenting the upstairs category $\mathcal W(T^*\tilde M)$ using pulled-back Floer data, we can push this isomorphism back down to $\mathcal W(T^*M)$ to conclude
    \[
        L \cong K^* \otimes T_q^*M.
    \]
    
    The authors expect the same to hold if $\pi_1(M)$ is infinite, but that requires extending Theorem \ref{thm: Lagrangian_homology_twisted} to the non-compact case.
\end{remark}
	
As we have seen, any Weinstein subdomain $X \subset T^*M_{std}$ induces a localization (Viterbo) functor $\Tw\mathcal{W}(T^*M_{std}; \mathbb{Z}) \rightarrow \Tw\mathcal{W}(X; \mathbb{Z})$ and hence by Remark \ref{rem: split_closure} is associated to a split-closed subcategory of 
$\Tw\mathcal{W}(T^*M_{std}; \mathbb{Z})$. 
So Theorem \ref{thm:Wein_subdom_classification} can be viewed as a classification of the split-closed subcategories of $\Tw\mathcal{W}(T^*M_{std}; \mathbb{Z})$ coming from this geometric setting. 
The fact that these correspond to subsets of prime integers stems from the corresponding fact for $\Tw\mathbb{Z}$ (and the crucial Corollary \ref{cor: disks_cotangent_bundle_image}). More generally, Hopkins and Neeman 
\cite{neeman_chromatic} proved that split-closed subcategories of $D^b\mathrm{Mod}_R$ correspond to certain subsets of $\mathrm{Spec}(R)$; in the global setting, Thomason \cite{thomason_thick_subcategories} proved that split-closed subcategories of $D^b \mathrm{Coh}(X)$ that are closed under the tensor product correspond to certain closed subsets of $X$. Although the wrapped Fukaya category does not generally have a monoidal structure, we pose the open problem of classifying Fukaya categories of Weinstein subdomains of arbitrary Weinstein domains as a way of extending these results to the symplectic setting.

\subsection{Acknowledgements}

We would like to thank Mohammed Abouzaid and Paul Seidel for helpful discussions, particularly concerning Proposition \ref{prop: C* modules homotopic}. The first author was partially supported by an NSF postdoctoral fellowship, award 1705128; the second author was partially supported by the Simons Foundation through grant \#385573, the Simons Collaboration on Homological Mirror Symmetry.

\section{Proof of results}\label{sec: proof_of_results}
	
\subsection{Constructing Lagrangian disks}	\label{subsection: constructing_Lag_disks}

Our construction of Weinstein subdomains of a Weinstein domain $X^{2n}$ depends on the existence of certain Lagrangian disks near the index $n$ co-cores of  $X^{2n}$. Since a neighborhood of an index $n$ co-core is $T^*D^n$, it suffices to construct these Lagrangians in $T^*D^n$. In this section, we will exhibit these Lagrangian disks in $T^*D^n$ and study their isomorphism classes in the partially wrapped category $\Tw\mathcal{W}(T^*D^n, \partial D^n; \mathbb{Z})$. 

Recall that objects of $\Tw\mathcal{W}(T^*D^n, \partial D^n; \mathbb{Z})$ are twisted complexes of exact Lagrangians in $T^*D^n$ whose boundary is disjoint from $\partial D^n$. We use the canonical $\mathbb{Z}$-grading of $T^*D^n$ via the Lagrangian fibration by cotangent fibers. By \cite{ganatra_generation, chantraine_cocores_generate}, the category  $\Tw\mathcal{W}(T^*D^n, \partial D^n)$ is generated by the cotangent fiber $T^*_0 D^n \subset T^*D^n$ at the origin ${0\in D^n}$.
$\mathrm{end}^*(T^*_0 D^n, T^*_0 D^n)$, the	partially wrapped Floer cochains of $T^*_0D^n$, is quasi-isomorphic to $\mathbb{Z}$, hence there is a cohomologically full and faithful $A_\infty$-functor 
\[
	CW^*(T^*_0 D^n, \_\,)\colon \Tw\mathcal{W}(T^*D^n, \partial D^n) \rightarrow \mathrm{Mod}_\mathbb{Z}
\]
Here $\mathrm{Mod}_\mathbb{Z}$ denotes the dg-category of right $\mathbb{Z}$-modules. Since $T^*_0 D^n$ generates the partially wrapped Fukaya category, this functor has image $\Tw\mathbb{Z}$, the category of cochain complexes whose underlying graded Abelian group is free and finitely generated. The equivalence between  $\Tw\mathcal{W}(T^*D^n, \partial D^n)$ and $\Tw\mathbb{Z}$ takes an object $L$ of $\Tw\mathcal{W}(T^*D^n, \partial D^n)$ to $CW^*(T^*_0 D^n, L)$, viewed as cofibrant cochain complex over $\mathbb{Z}$. 

Let $D^n_- \subset T^*D^n$ be a negative perturbation of the zero-section $D^n$, i.e. the result of applying the negative wrapping $\sum q_i\partial_{p_i}$ to $D^n$ so that $\partial D^n_-$ is disjoint from the stop $\partial D^n$. Note that $D^n_-$ is Lagrangian isotopic to $T^*_0 D^n$ in the complement of $\partial D^n$ by geodesic flow. Hence $CW^*(T^*_0 D^n, L)$ is quasi-isomorphic to $CW^*(D^n_-, L)$. Since all Reeb chords out of $\partial D^n_-$ hit the stop in small time, $CW^*(D^n_-, L)$ is quasi-isomorphic to $CF^*(D^n_-, L)$, the \textit{unwrapped} Floer cochains which can be explicitly computed.

Next we review certain regular Lagrangian disks in $T^*D^n$ introduced by Abouzaid and Seidel in Section 3b of \cite{abouzaid_seidel_recombination} and study their isomorphism class in $\Tw\mathcal{W}(T^*D^n, \partial D^n)$. Let $U \subset S^{n-1}$ be a compact codimension zero submanifold with smooth boundary. Let $g: S^{n-1} \rightarrow \mathbb{R}$ be a $C^1$-small function  so that $g$ is strictly negative in the interior of $U$, zero on $\partial U$,  strictly positive on $S^{n-1} \backslash U$, and has zero as a regular value. Next, extend $g$ to a smooth function $f: \mathbb{R}^n \rightarrow \mathbb{R}$ so that $f$ is $C^0$-small in the unit disk and satisfies $f(tq) = |t|^2 f(q)$ for $|q| \ge 1/2, t \ge 1$. Let $\Gamma(df)$ be the graph of $df$ in $T^*\mathbb{R}^n$ and let $D_U = \Gamma(df) \cap T^*D^n$. Since $f$ is homogeneous for $|q| \ge 1/2$ and $0$ is a regular value of $g$, $D_U$ has Legendrian boundary which is disjoint from $\partial D^n$. Furthermore, there is a Lagrangian isotopy $\Gamma(d(sf))$ from $D_U$ to the zero-section $D \subset T^*D^n$ (which intersects the stop $\partial D$ precisely when $s=0$). After fixing a grading on $D$, the isotopy $\Gamma(d(sf))$ induces a preferred grading on $D_U$. In particular, $D_U$ with this $\mathbb{Z}$-grading  is an object of $\Tw\mathcal{W}(T^*D^n, \partial D^n)$. 

We now compute the isomorphism class of $D_U$ in $\Tw\mathcal{W}(T^*D^n, \partial D^n)$, following \cite{abouzaid_seidel_recombination}. Namely, as noted in Lemma 3.3 of \cite{abouzaid_seidel_recombination}, we can scale $f$ so that the intersection points of $ D^n_U$ and $D^n_-$ have small action and then by a classical computation of Floer, $CW(D_-, D_U)$ is quasi-isomorphic to Morse cochains of $f$. Since $\mathbb{R}^n$ is contractible, this is quasi-isomorphic to 	 $\tilde{C}^{*-1}(U)$, reduced Morse cochains on $U$. Hence, under the equivalence $CW^*(T^*_0 D^n, \_\,)$ between $\Tw\mathcal{W}(T^*D^n, \partial D; \mathbb{Z})$ and $\Tw\mathbb{Z}$, the image of the disk $D_U$ in $\mathrm{Mod}_\mathbb{Z}$ is quasi-isomorphic to $\tilde{C}^{*-1}(U)$. Note that since $CW^*(T_0^* D^n, T_0^* D^n) \cong \mathbb{Z}$, the image of the twisted complex $\tilde{C}^{*-1}(U) \otimes T^*_0 D^n$ under the functor $CW^*(T^*_0 D^n, \_\,)$ is also $\tilde{C}^{*-1}(U)$. Since the $CW^*(T_0^*D^n, \_\,)$ functor is cohomologically full and faithful, the disk $D_U$ is quasi-isomorphic to the twisted complex $\tilde{C}^{*-1}(U) \otimes D_- \cong \tilde{C}^{*-1}(U) \otimes T^*_0 D^n$ in $\Tw\mathcal{W}(T^*D^n, \partial D; \mathbb{Z})$.

\begin{remark}
	Our definition of the disk $D_U$ agrees with that in Abouzaid-Seidel \cite{abouzaid_seidel_recombination}. However, they make an inconsequential misidentification of the Floer complex with the Morse complex to obtain $\tilde C^{*-1}(U)$ as $CF^*(D_U, D)$ instead of $CF^*(D, D_U)$.
\end{remark}

	Using the disks $D_U$, we now show that for sufficiently large $n$ any Lagrangian in $T^*D^n$ (or twisted complex of Lagrangians) is quasi-isomorphic to a Lagrangian disk. Note that this is stronger than the statement that any Lagrangian is a twisted complex of disks, which follows from the fact that  $T^*_0 D^n$ generates $\Tw\mathcal{W}(T^*D^n, \partial D^n; \mathbb{Z})$. 
	\begin{proposition}\label{prop: disks_in_cotangent_disk}
		If $n \ge 5$, every object of $\Tw\mathcal{W}(T^*D^n, \partial D^n; \mathbb{Z})$ is quasi-isomorphic to an exact Lagrangian disk. In particular, $\mathcal{W}(T^*D^n, \partial D^n; \mathbb{Z}) \rightarrow \Tw \mathcal{W}(T^*D^n, \partial D^n; \mathbb{Z})$ is a quasi-equivalence. 
	\end{proposition}

	\begin{proof} 
		An arbitrary object of 
		$\Tw \mathcal{W}(T^*D^n, \partial D^n; \mathbb{Z})$ can be identified
		with some finite dimensional cochain complex of free Abelian groups via the quasi-equivalence $CW^*(T^*_0 D^n, \_\,)$. 
		Every such cochain complex $(C^*, \partial)$ splits as a direct sum of twisted complexes of the form 
		$\mathbb{Z}[d+1] \overset{m}{\rightarrow}\mathbb{Z}[d]$ for some integer $m$ or complexes $\mathbb{Z}[d]$ with no differential. To see this,
		we use the fact that  the short exact sequence $0 \rightarrow \ker  \ \partial_n \rightarrow C_n \rightarrow \mathrm{im} \partial_n \rightarrow 0$ splits since $\mathrm{im} \partial_n$ is free. 

		Next, we recall that given two exact Lagrangians $L, K \subset X$ and a framed isotropic arc between their Legendrian boundaries $\partial L, \partial K \subset \partial X$, one can form a new exact Lagrangian $L \natural K \subset X$, the isotropic boundary connected sum of $L, K$. If $L, K$ are $\mathbb{Z}$-graded Lagrangians, then there is a choice of framing for the isotropic arc (the space of such choices up to homotopy is a $\mathbb{Z}$-torsor) so that $L\natural K$ also has a $\mathbb{Z}$-grading that restricts to the $\mathbb{Z}$-grading of $L$ and $K$ and hence $L\natural K$ is quasi-isomorphic to $L \oplus K$ in $\Tw\mathcal{W}(X)$; whenever we discuss the isotropic connected sum of two Lagrangians we mean the sum using any isotropic arc with this framing. The actual geometric disk will depend on the homotopy class of the arc, but since we are only concerned with the resulting object of $\mathcal W(X)$ we will ignore the distinction.

        Returning to $X = T^*D^n$, note that we can assume that  any two Lagrangians $L, K \subset T^*D^n$ are disjoint since we can view $T^*D^n$ as the result of gluing two copies of $T^*D^n$ together and place $L$ in one copy and $K$ in the other copy. So in light of the above discussion and the splitting from the previous paragraph, it suffices to prove that the twisted complexes $\mathbb{Z}[d+1] \overset{m}{\rightarrow} \mathbb{Z}[d]$ and $\mathbb{Z}[d]$ are quasi-isomorphic to embedded Lagrangian disks. The latter complex is quasi-isomorphic to $T^*_0 D^n$ with the appropriate grading so it suffices to prove that $\mathbb{Z}[d+1] \overset{m}{\rightarrow}  \mathbb{Z}[d]$ is quasi-isomorphic to a disk.

		As noted in Abouzaid-Seidel \cite{Abouzaid_Seidel}, for $n \ge 5$
		and any $m \ge 0$, there is a codimension 0 Moore space $U_m \subset S^{n-1}$ with $\tilde{C}^*(U_m) \cong \mathbb{Z}[-1] \overset{m}{\rightarrow} \mathbb{Z}[-2]$. For example, consider the CW complex $V$ obtained by attaching $D^2$ to $S^1$ along a degree $m$ map $S^1 \rightarrow S^1$; then $V$ embeds into $S^{n-1}$ for $n \ge 6$ by the Whitney trick and for $n =5$
		by the explicit map $D^2 \rightarrow \mathbb{C}^2$ given by $z\rightarrow ((1-|z|^2)z, z^m)$. Let $U_m$ be a neighborhood of $V$ in $S^{n-1}$. 
		Then $D_{U_m}$  
		is quasi-isomorphic to $\mathbb{Z}[-2] \overset{m}{\rightarrow} \mathbb{Z}[-3]$. Finally, we shift the grading on $D_{U_m}$ by $d+3$ and the resulting disk  $D_{U_m}[d+3]$ is quasi-isomorphic to 
		$\mathbb{Z}[d+1] \overset{m}{\rightarrow} \mathbb{Z}[d]$,   as desired.
	\end{proof}
	We observe that not every object of $\Tw\mathcal{W}(T^*D^n, \partial D^n; \mathbb{Z})$ is quasi-isomorphic to a disk $D_U$. This is because $CW^*(D_-, D_U)$ is a cochain complex that is supported between degrees $0$ and $n-1$ (since $U \subset S^{n-1}$) or a shift thereof (if we shift the grading on $D^n_U$) while  a general cochain complex can have arbitrarily wide support. However, Proposition \ref{prop: disks_in_cotangent_disk} shows that every object of $\Tw\mathcal{W}(T^*D^n, \partial D^n; \mathbb{Z})$ is quasi-isomorphic to the boundary connected sum of possibly several different $D_U$, with possibly different gradings. 	
	
	\subsection{Constructing subdomains}
	
	Now we use the Lagrangian disks from the previous section to construct Weinstein subdomains of a Weinstein domain $X$ and prove Theorem \ref{thm: exotic_subdomains}. 
	As stated in Remark \ref{rem: stops},  the construction of subdomains also holds  when the ambient Weinstein domains has stops. 	The most important case for us is when $X = (T^*D^n, \partial D^n)$, the stopped domain considered in the previous section. As we will see, Theorem \ref{thm: exotic_subdomains} for arbitrary Weinstein domains follows from this case.
	
	In the following, we say a stopped Weinstein domain $(X_0, \Lambda_0)$ is a Weinstein subdomain of $(X, \Lambda)$ if $X = X_0 \cup C$ for some Weinstein cobordism $C$ which is trivial along $\Lambda_0 = \Lambda$. In particular, there is a smoothly trivial regular Lagrangian cobordism between $\Lambda_0$ and $\Lambda_1$ in $X\backslash X_0$ which allows to identify the linking disk of $\Lambda_0$ in $X_0$ with the linking disk of $\Lambda$ in $X$. We say that this cobordism is flexible if the attaching spheres of the index $n$ handles are loose in the complement of $\Lambda_0$. We also say that two Weinstein subdomains $X_0, X_1 \subset X$ are Weinstein homotopic if the following holds: there is a homotopy of Weinstein Morse functions $f_t, 0\le t \le 1$, on $X$ that have $c$ as a regular level set for all $t$ and $X_0$ and $X_1$ are the $c$-sublevel sets of $f_0$ and $f_1$, respectively.

\begin{theorem}\label{thm: p-handles}
		Let $n \ge 5$. For any finite collection of prime numbers $P$, that is possibly empty or contains $0$, 
		there is a Legendrian sphere $\Lambda_P \subset \partial B^{2n}_{std}$ formally isotopic to the standard unknot $\Lambda_\varnothing$ so that $(B^{2n}_{std}, \Lambda_P)$ embeds as a Weinstein subdomain of $(B^{2n}_{std}, \Lambda_{\varnothing}) = (T^*D^n, \partial D^n)$ with the following properties: 
		\begin{enumerate}
			\item The Viterbo restriction functor
			\[
			\Tw\mathcal{W}(B^{2n}_{std}, \Lambda_\varnothing; \mathbb{Z}) \rightarrow
			\Tw\mathcal{W}(B^{2n}_{std}, \Lambda_P; \mathbb{Z})
			\]
			induces an equivalence
			\[
			\Tw\mathcal{W}(B^{2n}_{std}, \Lambda_P; \mathbb{Z}) \cong \Tw\mathcal{W}(B^{2n}_{std}, \Lambda_\varnothing; \mathbb{Z})\left[\frac{1}{P}\right] \cong \Tw\mathbb{Z}\left[\frac{1}{P}\right].
			\]
			\item $(B^{2n}_{std}, \Lambda_P)$ embeds as a Weinstein subdomain of $(B^{2n}_{std}, \Lambda_Q)$ if and only if $Q \subset P$ or $0 \in P$. In such cases, we can construct such an embedding with the property that the Weinstein cobordism between unstopped domains is trivial, i.e. $\partial B^{2n}\times[0,1]$, and if $R \subset Q \subset P$, the composition $(B^{2n}_{std}, \Lambda_P) \subset (B^{2n}_{std}, \Lambda_Q) \subset (B^{2n}_{std}, \Lambda_R)$ is Weinstein homotopic to 
			$(B^{2n}_{std}, \Lambda_P) \subset (B^{2n}_{std}, \Lambda_R)$ obtained by viewing $P \subset R$. 
			\item There is a smoothly trivial regular Lagrangian cobordism $L \subset \partial B^{2n}_{std}\times [0,1]$ with $\partial_- L = \Lambda_P$ and $\partial_+ L =\Lambda_Q$ if and only if  $Q \subset P$ or $0 \in P$.
			Furthermore, for two disjoint subsets of primes  $P_1, P_2$, the Legendrian sphere
			$\Lambda_{P_1 \amalg P_2}$ is the isotropic connected sum $\Lambda_{P_1} \natural \Lambda_{P_2}$ of $\Lambda_{P_1}, \Lambda_{P_2}$ embedded
			in disjoint Darboux balls in $\partial B^{2n}_{std}$. 
			
			\item  If $0 \in P$, then $\Lambda_P \subset \partial B^{2n}_{std}$ is loose.
		\end{enumerate}
	\end{theorem}
	In particular, we have a sequence of Legendrians 
	$$\Lambda_{unknot} = \Lambda_\varnothing, \Lambda_2, \Lambda_{2,3}, \Lambda_{2,3,5},\Lambda_{2,3,5,7}, \dotsc, \Lambda_0 = \Lambda_{loose}
	$$
	in $\partial B^{2n}_{std}$
	and Lagrangian cobordisms in $\partial B^{2n}_{std} \times [0,1]$ connecting consecutive Legendrians interpolating between $\Lambda_{unknot}$ and $\Lambda_{loose}$, analogous to the sequence of subdomains in Theorem \ref{thm: exotic_subdomains}.
	We note that such Legendrians do not exist for $n = 2$ as proven in \cite{cornwell_ng_sivek_Lagrangian_concordance}: if $L^2$ is a \textit{decomposable} Lagrangian cobordism (a condition similar to regularity) with negative end $\Lambda$ and positive end $\Lambda_\varnothing$, then either $\Lambda = \Lambda_\varnothing$ or $\Lambda$ is stabilized in the sense of \cite{Murphy11},  so $\Tw\mathcal{W}(B^4_{std},  \Lambda; \mathbb{Z}) \cong \Tw\mathbb{Z}$ or $\Tw\mathcal{W}(B^4_{std},  \Lambda; \mathbb{Z}) \cong 0$ are the only possibilities.
	\begin{remark}
	The construction of the isotropic connected sum $\Lambda_1 \natural \Lambda_2$ of two Legendrians $\Lambda_1, \Lambda_2$ in the statement of Theorem \ref{thm: p-handles} above 
	is similar to the boundary connected sum of two Lagrangians discussed in the proof of Proposition \ref{prop: disks_in_cotangent_disk} 
	(the former happens on the boundary of the latter) 
	and also depends on a framed isotropic arc between $\Lambda_1$ and $\Lambda_2$.
	However if $\Lambda_2$ is contained in a Darboux chart of $(Y, \xi)$ disjoint from $\Lambda_1$, then the isotropic connected sum $\Lambda_1 \natural \Lambda_2$ is actually independent of the isotropic arc and its framing; this is because we can isotope $\Lambda_2$ to a small neighborhood of $\Lambda_1$ via the original isotropic arc and then isotope it back to its original position using a new isotropic arc.
	
	More precisely, we can identify the Darboux chart containing $\Lambda_2$ with the cotangent bundle of a small piece of the framing-thickened arc and use this to produce a family of Darboux charts. If the two arcs have the same framing at the endpoints, the resulting family of Darboux charts is a loop, which means that $\Lambda_2$ returns to itself.
	\end{remark}
	\begin{proof}[Proof of Theorem \ref{thm: p-handles}]
		We prove this theorem in several stages: first we construct $\Lambda_p$ 
		when $p$ is a single prime and prove that it has the claimed geometric properties,  then we 
		construct $\Lambda_P$ for a general set of primes $P$, and finally we prove our claims about the Fukaya category of $(B^{2n}_{std}, \Lambda_P)$. 
		
		\subsubsection{$\Lambda_p$ for a single prime $p$} 
		
		We first consider the case when the collection of primes $P$ consists of a single prime $p$.
		As discussed in the previous section, let $U_p \subset S^{n-1}$ be a fixed $p$-Moore space. Then the Lagrangian disk $D_p: = D_{U_p} \subset (T^*D^n, \partial D^n)$ is isomorphic to
		$T^*_0 D^n[1]\overset{p}{\rightarrow} T^*_0 D^n$ in $\Tw \mathcal{W}(T^*D^n, \partial D^n; \mathbb{Z})$. 
		Also, if $p = 0$, we set $U = S^{n-1}$  (as a full subset of $S^{n-1}$) and form 
		$D_0 := D^n_{S^{n-1}}$, which is Lagrangian isotopic in $(T^*D^n, \partial D^n)$ to the the cotangent fiber $T^*_0 D^n$. If $p$ is the empty set, we set 
		$U = B^{n-1}\subset S^{n-1}$ and form $D_\varnothing := D^n_{B^{n-1}}$, which is a small Lagrangian disk that is disjoint from the zero-section $D^n \subset T^*D^n$; note that any two such small Lagrangian disks are isotopic in $(T^*D^n, \partial D^n)$. In particular, $D_\varnothing$ is the zero object in $\Tw \mathcal{W}(T^*D^n, \partial D^n; \mathbb{Z})$. 
		To construct $(B^{2n}_{std}, \Lambda_P)$, we will \textit{carve out} these Lagrangian disks as we now explain.
		
		In general, given a Liouville domain $X^{2n}$ and an exact Lagrangian disk $D^n \subset X^{2n}$ with Legendrian boundary, there is a Liouville subdomain $X_0 \subset X$ (which we say is obtained by \textit{carving out} $D^n$ from $X$) and a Legendrian sphere $\Lambda \subset \partial X_0$ so that $X= X_0 \cup  H^n_{\Lambda}$ and the co-core of $H^n_{\Lambda}$ is $D^n$; see \cite{EGL} for details. If $X$ is a Weinstein domain and $D^n \subset X$ is a regular Lagrangian, then $X_0 \subset X$ is a Weinstein subdomain. The disks $D_p \subset T^*D^n$ we consider are indeed regular; in fact $D_p = \Gamma(df)$ is isotopic through Lagrangians with Legendrian boundary $(D_p)_s = \Gamma(sdf) \cap T^*D^n$ to the zero-section $D^n \subset T^*D^n$. Therefore,  $T^*D^n \backslash D_p$ is homotopic to the Weinstein domain $T^*D^n \backslash D^n$, which is actually the subcritical domain $T^*(S^{n-1} \times D^1) = B^{2n}_{std} \cup H^{n-1}$. Since $D_p$ is disjoint from $\partial D^n$, we can consider $(T^*D^n, \partial D^n) \backslash D_p$ as $T^*(S^{n-1}\times D^1)$ with some stop, namely the image of $\partial D^n$. 
		
		Since the subdomain $T^*D^n\backslash D_p$ is obtained by carving out $D_p$,  there is a Legendrian $\Lambda \subset \partial (T^*D^n \backslash D_p)$ disjoint from $\partial D^n$ so that $(T^*D^n, \partial D^n) \backslash D_p \cup H^n_\Lambda = (T^*D^n, \partial D^n)$ and the co-core of $H^n_\Lambda$ is $D_p$. Because $n \ge 5$, there is a unique loose Legendrian $\Lambda_{loose} \subset \partial (T^*D^n \backslash D_p)$ that is formally isotopic to $\Lambda$ and is loose in the complement of $\partial D^n$; see \cite{Murphy11}. Next we form the stopped domain $(T^*D^n, \partial D^n) \backslash D_p \cup H^n_{flex}$ by attaching the handle $H^n_{flex}$ along $\Lambda_{loose}$. 
		We note that the ambient Weinstein domain $T^*D^n \backslash D_p \cup H^n_{flex}$ is flexible since $T^*D^n \backslash D_p$ is subcritical and $H^n_{flex}$ is attached a loose Legendrian.
		Furthermore, it is formally symplectomorphic to the standard Weinstein ball since $\Lambda_{loose}$ is formally isotopic to $\Lambda$ (and attaching a handle to $\Lambda$ reproduces $B^{2n}_{std}$). Therefore by the h-principle for flexible Weinstein domains \cite{CE12}, $T^*D^n \backslash D_p \cup H^n_{flex}$  is Weinstein homotopic to $B^{2n}_{std}$. Under this identification with $B^{2n}_{std}$, the stop  $\partial D^n  \subset T^*D^n \backslash D_p \cup H^n_{flex}$ 
		becomes some Legendrian in $\partial B^{2n}_{std}$ which we call $\Lambda_p$. That is, we set
		$$
		(B^{2n}_{std}, \Lambda_p) := (T^*D^n, \partial D^n) \backslash D_p \cup H^n_{flex}
		$$
		We will show that $(B^{2n}_{std}, \Lambda_p)$ satisfies the claimed properties.

		First we show that
		$(B^{2n}_{std}, \Lambda_p)$ is a Weinstein subdomain of $(T^*D^n, \partial D^n)$. Note that $(B^{2n}_{std}, \Lambda_\varnothing)$ is precisely $(T^*D^n, \partial D^n)$. This is because $(T^*D^n, \partial D^n) \backslash D_\varnothing = (T^*D^n, \partial D^n) \cup H^{n-1}$ and the Legendrian $\Lambda$ from the previous paragraph intersects the belt sphere of $H^{n-1}$ exactly once; so $H^{n-1} \cup H^n_{flex}$ are cancelling handles and hence 
		$$
		(T^*D^n, \partial D^n) \backslash D_\varnothing  \cup H^{n}_{flex} = (T^*D^n, \partial D^n) \cup H^{n-1}\cup H^n_{flex} = (T^*D^n, \partial D^n)
		$$
		Now we consider the case when $p$ is a (non-zero) prime.
		It is clear that $(T^*D^n, \partial D^n) \backslash D_p$ is a subdomain of $(T^*D^n, \partial D^n)$ by construction; we claim that it is still a subdomain even after attaching the flexible handle $H^n_{flex}$ to $(T^*D^n, \partial D^n) \backslash D_p$. To see this, let $C$ be the  Weinstein cobordism between  $(T^*D^n, \partial D^n) \backslash D_p$ and $(T^*D^n, \partial D^n)$ given by the handle $H^n_\Lambda$ (whose co-core is $D_p$). 
		By \cite{Lazarev_crit_points}, we can Weinstein homotope $C$, in the complement of $\partial D^n$, to a Weinstein cobordism $H^n_{flex} \cup H^{n-1}\cup H^n_{\Lambda'}$, where $H^n_{flex}$ is attached along $\Lambda_{loose}$ and
		$H^{n-1} \cup H^n_{\Lambda'}$ is a smoothly trivial Weinstein cobordism whose attaching spheres are disjoint from $\partial D^n$. 
		So we have the following equalities, up to Weinstein homotopy:
		\begin{eqnarray}
		(B^{2n}_{std}, \Lambda_p) \cup H^{n-1}\cup H^n_{\Lambda'} &=& 
		(T^*D^n, \partial D^n) \backslash D_p \cup H_{flex} \cup H^{n-1}\cup H^n_{\Lambda'}\\
		&=&  (T^*D^n, \partial D^n) \backslash D_p \cup C = (T^*D^n, \partial D^n)
		\end{eqnarray}
		which show that $(B^{2n}_{std}, \Lambda_p)$ 
		is a subdomain of $(B^{2n}_{std}, \Lambda_\varnothing) = (T^*D^n, \partial D^n)$.
		Furthermore, the construction in \cite{Lazarev_crit_points} shows that $\Lambda'$ is loose (but not in the complement of $\Lambda_{loose}$ or $\partial D^n$) since $\Lambda$ is loose (but not in the complement of $\partial D^n$). So the Weinstein cobordism $H^{n-1}\cup H^n_{\Lambda'}$ is flexible (but not in the complement of the stop $\partial D^n$) and therefore is homotopic to $\partial B^{2n}_{std}\times[0,1]$.
		
		Since the attaching spheres of $H^{n-1}, H^n_{\Lambda'}$ are disjoint from $\partial D^n$,  we can view $\partial D^n \times [0,1]$ as a trivial Lagrangian cobordism between $\partial D^n$ in $T^*D^n \backslash D_p \cup H^n_{flex}$ and $\partial D^n$ in $T^*D^n$. Under our identifications, this produces a smoothly trivial regular Lagrangian cobordism (regular in that the Liouville vector field can be made tangent to it) between $\Lambda_p$ and $\Lambda_\varnothing$ in $\partial B^{2n}_{std}\times [0,1]$ as desired.
		We also observe that $\Lambda_p$ is formally Legendrian isotopic to $\Lambda_\varnothing$ in $\partial B^{2n}_{std}$ because the attaching spheres $\Lambda$ and $\Lambda_{loose}$ are formally Legendrian isotopic in the complement of $\partial D^n$. More precisely, note that
		$\partial D^n \amalg \Lambda$ and $\partial D^n \amalg \Lambda_{loose}$ are formally isotopic Legendrian links. 
		Furthermore, there is a genuine Legendrian isotopy from
		$\Lambda$ to $\Lambda_{loose}$ (but not in the complement of $\partial D^n$) and so this extends to a Legendrian isotopy from 
		$\partial D^n \amalg \Lambda_{loose}$ to $\overline{\partial D^n} \amalg \Lambda$, where $\overline{\partial D^n}$ 
	is some other Legendrian that becomes $\Lambda_p$ after handle attachment to $\Lambda$.  Since a genuine Legendrian isotopy preserves formal Legendrian isotopies, 
		$\partial D^n \amalg \Lambda$ and $\overline{\partial D^n} \amalg \Lambda$ are also formally Legendrian isotopic links. So when we attach a handle to $\Lambda$ to get $B^{2n}_{std}$, $\partial D^n$ and $\overline{\partial D}^n$ are still formally Legendrian isotopic in $\partial B^{2n}_{std}$, which is precisely the statement that $\Lambda_\varnothing, \Lambda_p$ are formally Legendrian isotopic.

		Next we consider the case when $p = 0$. Recall that in this case, $D_0$ is the cotangent fiber $T^*_0 D^n \subset T^*D^n$. Then $(T^*D^n, \partial D^n) \backslash D_p$ is $(T^*(S^{n-1} \times D^1), S^{n-1} \times \{0\})$. We note that $S^{n-1} \times \{0\} \subset \partial T^*(S^{n-1} \times D^1)$ is loose since this Legendrian crosses the belt sphere of the index $n-1$ handle (corresponding to the index $n-1$ Morse critical point of $S^{n-1}$) exactly once; see \cite{CE12} for this looseness criterion. 
		To construct $(B^{2n}_{std}, \Lambda_0)$ from  $(T^*D^n, \partial D^n) \backslash D_p = (T^*(S^{n-1} \times D^1), S^{n-1} \times \{0\})$, we attach an index $n$ handle $H^n_{flex}$ along the Legendrian $\Lambda_{loose}$ that is loose in the complement of $\partial D^n = S^{n-1}\times \{0\}$ (and is formally isotopic to the Legendrian $\Lambda$).
		Since $\partial D^n = S^{n-1}\times \{0\}$ is loose and $\Lambda_{loose}$ is loose in the complement of $\partial D^n$, then $\partial D^n = S^{n-1}\times \{0\}$ is in fact also loose in the complement of $\Lambda_{loose}$, i.e. $\partial D^n = S^{n-1}\times \{0\}$ and $\Lambda_{loose}$ form a loose link; see \cite{CE12} for an argument explaining this fact.
		In particular, the loose chart of $\partial D^n$ persists under attaching the handle $H^{n}_{flex}$ along $\Lambda_{loose}$ and so $\partial D^n \subset (T^*D^n, \partial D^n) \backslash D_p \cup H^n_{flex}$ is still loose. By definition, this means that the stop $\Lambda_0$ in $(B^{2n}_{std}, \Lambda_0)$ is loose as desired. As in the previous paragraph, $(B^{2n}_{std}, \Lambda_0)$ is a Weinstein subdomain of $(B^{2n}_{std}, \Lambda_\varnothing)$. Hence there is a regular Lagrangian cobordism from a loose Legendrian to the Legendrian unknot, as originally proven in \cite{EM, Lazarev_reg_lag}.
		
		This proves all of claims (2), (3), (4) when $P$ consists of a single element.

		\subsubsection{$\Lambda_P$ for a collection of primes $P$}
		
		Now we construct $\Lambda_P$ when $P=\{p_1, \dotsc, p_k\}$ is a collection of primes with multiple elements.
		We consider disjoint Weinstein balls $B^{2n}_{std,i}$ so that $\Lambda_{p_i}\subset \partial B^{2n}_{std,i}$
		and do a \textit{simultaneous} boundary connected sum to the $B^{2n}_{std,i}$ and $\Lambda_i$,
		as in the construction of regular Lagrangians \cite{EGL}:
		$$
		(B^{2n}_{std}, \Lambda_P):= (B^{2n}_1, \Lambda_{p_1}) \natural \dotsb \natural (B^{2n}_k, \Lambda_{p_k})
		$$
		Namely, we attach index 1 Weinstein handles to the disjoint union of Weinstein balls  $$B^{2n}_{std, 1} \amalg \dotsb \amalg B^{2n}_{std, k}$$ so that the attaching spheres of these index 1 handles, i.e. two points, are on different $\Lambda_i$; we simultaneously do Legendrian surgery on the $\Lambda_i$ via isotropic arcs in the 1-handles. The resulting Legendrian $\Lambda_P$ is connected  and in fact coincides with the usual  isotropic connected sum of Legendrians $\Lambda_{p_1}, \dotsc, \Lambda_{p_k}$ embedded in disjoint Darboux balls in a \textit{single}
		$\partial B^{2n}_{std}$.
		This also shows that up to Legendrian isotopy, $\Lambda_P$ does not depend on the order of the set $P$.
		
		Next we show that $(B^{2n}_{std}, \Lambda_P)$ is a Weinstein subdomain of
		$(B^{2n}_{std}, \Lambda_Q)$ if $Q \subset P$.
		Via our previous identification, $(B^{2n}_{std}, \Lambda_P)$ is the same as
		\begin{equation}\label{eqn: several_primes_carve_out}
		((T^*D^n, \partial D^n)\backslash D_{p_1} \cup H^n_{flex}) \natural \dotsb \natural ((T^*D^n, \partial D^n)\backslash D_{p_k} \cup H^n_{flex}) 
		\end{equation}
		where we choose points on each $\partial D^n$ to do the simultaneous boundary connected sum. So if $Q \subset P$, $(B^{2n}_{std}, \Lambda_P)$ differs from $(B^{2n}_{std}, \Lambda_Q)$ by a boundary connected sum with
		\[
		(T^*D^n, \partial D^n)\backslash D_{p} \cup H^n_{flex}
		\]
		for all $p \in P \backslash Q$.
		We saw previously that 
		$
		(T^*D^n, \partial D^n)\backslash D_{p} \cup H^n_{flex}$ is a  subdomain 
		of $(T^*D^n, \partial D^n)$ and hence $(B^{2n}_{std}, \Lambda_P)$ is a subdomain of $(B^{2n}_{std}, \Lambda_Q)$ boundary connected sum with several copies of $(T^*D^n, \partial D^n)$, one for each $p \in P\backslash Q$. Since doing boundary connected sum with $(T^*D^n, \partial D^n)$ does not change the Weinstein homotopy type, the latter domain is still $(B^{2n}_{std}, \Lambda_Q)$ and so  $(B^{2n}_{std}, \Lambda_P)$ is a subdomain of $(B^{2n}_{std}, \Lambda_Q)$ as desired. 
		This also shows that if $R \subset Q \subset P$, the Weinstein cobordism $(B^{2n}_{std}, \Lambda_R) \setminus (B^{2n}_{std}, \Lambda_P)$ is homotopic to the concatenation of Weinstein cobordisms  $(B^{2n}_{std}, \Lambda_Q)\setminus (B^{2n}_{std}, \Lambda_P)$
		and  $(B^{2n}_{std}, \Lambda_R) \setminus (B^{2n}_{std}, \Lambda_Q)$. Since $(B^{2n}_{std}, \Lambda_P)$ is a subdomain of
		$(B^{2n}_{std}, \Lambda_Q)$, we have a Lagrangian cobordism in $\partial B^{2n}_{std}\times [0,1]$ with negative boundary  $\Lambda_P$ and positive boundary $\Lambda_Q$ by definition.

		If $0 \in P$, then $\Lambda_P$ is loose since it is the isotropic connected sum of $\Lambda_{P\backslash 0}$ and $\Lambda_0$, which we already saw to be loose. Let $Q$ be another set of primes. Then $\Lambda_{P}$ and $\Lambda_{P\cup Q}$ are both loose unknots (since $P \cup Q$ contains $0$) and so $\Lambda_{P}$ and $\Lambda_{P\cup Q}$ are Legendrian isotopic by the h-principle for loose Legendrians \cite{Murphy11}. By the previous discussion, this implies that $(B^{2n}_{std}, \Lambda_P) = (B^{2n}_{std}, \Lambda_{P \cup Q})$ is a subdomain of $(B^{2n}_{std}, \Lambda_Q)$ since now $Q \subset P \cup Q$.

		This proves all of claims (2), (3), (4), except the `only if' part of claim (2), (3).
		
		\subsubsection{Fukaya category of $(B^{2n}_{std}, \Lambda_P)$}
		Finally, we compute the partially wrapped Fukaya category of $(B^{2n}_{std}, \Lambda_P)$. By the description in Equation \ref{eqn: several_primes_carve_out}, $(B^{2n}_{std}, \Lambda_P)$ is the result of carving out the disks $D_{p_1}, \dotsc, D_{p_k}$ from $(B^{2n}_{std}, \Lambda_\varnothing) = (T^*D^n, \partial D^n)$ and then attaching some flexible handles; here the disks are embedded disjointly by viewing $(T^*D^n, \partial D^n)$ as the boundary connected sum of several disjoint copies of $(T^*D^n, \partial D^n)$.  
		By \cite{ganatra_generation, Sylvan_talks}, there is a geometrically defined Viterbo transfer functor 
		$$
		\Tw \mathcal{W}(T^*D^n, \partial D)  
		\rightarrow 
		\Tw \mathcal{W}((T^*D^n, \partial D) \backslash D_p)  
		$$
		which is localization by $D_p$. That is, 
		$
		\Tw \mathcal{W}((T^*D^n, \partial D) \backslash D_p)   \cong \Tw \mathcal{W}(T^*D^n, \partial D)/D_p
		$
		and the Viterbo functor is the algebraic localization by the object $D_p$.
		By construction, the Lagrangian  $D_p$ of $ \Tw \mathcal{W}(T^*D^n, \partial D^n; \mathbb{Z})$ is isomorphic to the twisted complex $$T^*_0 D^n[1]\overset{p}{\rightarrow} T^*_0 D^n = \mathrm{cone}(p \cdot \mathrm{Id}_{T^*_0 D^n}).$$ So 
		$
		\Tw \mathcal{W}((T^*D^n, \partial D) \backslash D_p)   \cong \Tw \mathcal{W}(T^*D^n, \partial D)/\mathrm{cone}(p \cdot \mathrm{Id}_{T^*_0 D^n})
		$
		Furthermore, the localization by a collection of objects depends only the split-closure of that collection of objects. Since $T^*D_0$ generates $\Tw \mathcal{W}(T^*D^n, \partial D^n)$,  we have the equivalence
		\begin{align*}\label{eqn: localization_isomorphism1}
		\Tw \mathcal{W}(T^*D^n &, \partial D)/\mathrm{cone}(p \cdot \mathrm{Id}_{T^*_0 D^n}) \\
		&\cong \Tw \mathcal{W}(T^*D^n, \partial D)/\{ \mathrm{cone}(p \cdot \mathrm{Id}_{L})| L \in \Tw \mathcal{W}(T^*D^n, \partial D^n) \}\\
		&=:\Tw \mathcal{W}(T^*D^n, \partial D; \mathbb{Z})\left[\frac{1}{p}\right]
		\end{align*}
		Combining with the previous equivalence, we have
		\begin{equation}
		\Tw \mathcal{W}((T^*D^n, \partial D)\backslash D_p; \mathbb{Z})
		\cong \Tw \mathcal{W}(T^*D^n, \partial D; \mathbb{Z})\left[\frac{1}{p}\right]
	     \end{equation}
		Similarly, when we carve out multiple disks $D_{p_1}, \dotsc, D_{p_k}$, we invert $p_1, \dotsc, p_k$ in the Fukaya category. 
		Attaching flexible handles does not affect the Fukaya category 
		and so 
		\begin{equation}\label{eqn: localization_isomorphism}
		\Tw \mathcal{W}(B^{2n}_{std}, \Lambda_P; \mathbb{Z})\cong 
		\Tw \mathcal{W}(T^*D^n, \partial D^n; \mathbb{Z})\left[\frac{1}{P}\right]
		\end{equation}
		as desired. If $p$ is zero, then $D_p = T^*_0 D^n$ and $\Tw \mathcal{W}(T^*D^n, \partial D)/ T^*_0 D^n \cong 0$, which is indeed the case for $(B^{2n}_{std}, \Lambda_0)$ since $\Lambda_0$ is loose. 
		
		\begin{remark}
			We note that the above discussion does not automatically show that that the equivalence in Equation \ref{eqn: localization_isomorphism} is given by the Viterbo functor induced by the Weinstein embedding of $(B^{2n}_{std}, \Lambda_P)$ into $(B^{2n}_{std}, \Lambda_\varnothing) = (T^*D^n, \partial D^n)$ due to the presence of the extra flexible handles. However this is indeed the case. Recall that the Weinstein cobordism between these two domains is 
			$H^{n-1} \cup H^n_{\Lambda'}$, which comes from a construction in \cite{Lazarev_crit_points, Lazarev_geometric_presentations}. The proof there shows that the co-core of $H^n_{\Lambda'}$ is $D_p \natural \overline{D_p} \subset (T^*D^n, \partial D^n)$ and so the Viterbo functor between these two domains is localization by $D_p \natural \overline{D_p}$. Now
			$D_p \natural \overline{D_p} \cong D_p \oplus D_p[1]$
			and $D_p$ have the same split-closure, and so  localization by $D_p \natural \overline{D_p}$ is the same as  localization by $D_p$, as in Equation \ref{eqn: localization_isomorphism}.
		\end{remark}

		Finally, we prove the `only if' part of claims (2), (3). Suppose that $(B^{2n}_{std}, \Lambda_P)$ is Weinstein subdomain of $(B^{2n}_{std}, \Lambda_Q)$ but $Q \not \subset P$ and $0 \not \in P$. There would be a localization functor from the Fukaya category of $(B^{2n}_{std}, \Lambda_Q)$ to that of 
		$(B^{2n}_{std}, \Lambda_P)$ over any coefficient ring $R$. However, if we take $R = \mathbb{F}_q$ for any $q \in Q \backslash P$, we have  $D_q \cong \mathrm{cone}(0_{T^*_0 D^n}) \cong T^*_0 D^n[1]\oplus  T^*_0 D^n$
		in $\Tw  \mathcal{W}(B^{2n}_{std}, \Lambda_\varnothing; \mathbb{F}_q)$ since $q \equiv 0$ in $\mathbb{F}_q$. 
		This object split-generates $\Tw  \mathcal{W}(B^{2n}_{std}, \Lambda_\varnothing; \mathbb{F}_q)$ and so 
		$$
		\Tw \mathcal{W}(B^{2n}_{std}, \Lambda_Q; \mathbb{F}_q) \cong \Tw \mathcal{W}(B^{2n}_{std}, \Lambda_\varnothing; \mathbb{F}_q)/D_q 
		\cong 0
		$$
		On the other hand, 
		all $p\in P$ are invertible in $\mathbb{F}_q$ because $q \in Q \backslash P$ by assumption and $p \ne 0$. Therefore 
		$D_p \cong \mathrm{cone}(p \cdot \mathrm{Id}_{T^*_0 D^n}) \cong 0$ in 
		$\Tw  \mathcal{W}(B^{2n}_{std}, \Lambda_\varnothing; \mathbb{F}_p)$ for all $p \in P$ and so 
		$$
		\Tw \mathcal{W}(B^{2n}_{std}, \Lambda_P; \mathbb{F}_q) \cong \Tw  \mathcal{W}(B^{2n}_{std}, \Lambda_\varnothing; \mathbb{F}_q)/0\cong \Tw  \mathcal{W}(B^{2n}_{std}, \Lambda_\varnothing; \mathbb{F}_q) \cong \Tw \mathbb{F}_q
		$$
		which is non-trivial. Since there cannot be a localization functor from the trivial category to $\Tw  \mathbb{F}_q$,  $(B^{2n}_{std}, \Lambda_P)$ cannot be a Weinstein subdomain of $(B^{2n}_{std}, \Lambda_Q)$. This proves the `only if' part of claim (2). If there is a smoothly trivial regular Lagrangian cobordism from $\Lambda_P$ to $\Lambda_Q$ in $\partial B^{2n}_{std} \times [0,1]$, then $(B^{2n}_{std}, \Lambda_P)$ is a Weinstein subdomain of $(B^{2n}_{std}, \Lambda_Q)$ and so the `only if' part of claim (3) follows from that for claim (2).
	\end{proof}
	
	Now we show that Theorem \ref{thm: p-handles} implies Theorem \ref{thm: exotic_subdomains} concerning Weinstein subdomains of an \textit{arbitrary} Weinstein domain.
	Recall that an index $n$ Weinstein handle can be viewed as the stopped domain $(T^*D^n, \partial D^n) = (B^{2n}_{std}, \Lambda_{\varnothing})$. We will consider the stopped domains $(B^{2n}_{std}, \Lambda_P)$ in Theorem \ref{thm: p-handles} as generalized Weinstein handles. 
	\begin{definition}
		A $P$-Weinstein handle of index $n$ is the stopped domain $(B^{2n}_{std}, \Lambda_P)$.
	\end{definition}
	Here our model for the $P$-Weinstein handle uses explicit embeddings of Moore spaces into $S^{n-1}$ and hence is well-defined.
	When attaching Weinstein handles, one implicitly uses the canonical parametrization of $\partial D^n \subset T^*D^n$. Via the construction in the proof of Theorem \ref{thm: p-handles}, this parametrization gives the Legendrians $\Lambda_P \subset \partial B^{2n}$ a parametrization as well. Therefore, given a parametrized Legendrian sphere $\Lambda$ in a contact manifold $(Y, \xi)$, we can attach a $P$-Weinstein handle $(B^{2n}_{std}, \Lambda_P)$ to it and produce a Weinstein cobordism, just like we do for usual Weinstein handles. 
	To prove Theorem \ref{thm: exotic_subdomains}, we replace all standard Weinstein $n$-handles $(B^{2n}_{std}, \Lambda_\varnothing)$ with Weinstein $P$-handles $(B^{2n}_{std}, \Lambda_P)$.  
	\begin{proof}[Proof of Theorem \ref{thm: exotic_subdomains}]
		Let $X^{2n}$ be a Weinstein domain with $n \ge 5$ and  $C_1^n, \dotsc, C_k^n \subset X^{2n}$ the Lagrangian co-core disks of its index $n$ handles $H^n_1, \dotsc, H^n_k$.
		Hence there is a subcritical Weinstein domain $X_0 \subset X$ and Legendrian spheres $\Lambda_1, \dotsc, \Lambda_k \subset \partial X_0$ so that $X = X_0 \cup H^n_{\Lambda_1} \cup \dotsb \cup H^n_{\Lambda_1}$ and the co-core of $H^n_{\Lambda_i}$ is $C_i \subset X$. That is, $X_0$ is obtained from $X$ by carving out the Lagrangian disks $C_1, \dotsc, C_k$. This gives the following decomposition of $X:$ 
		\begin{equation}\label{eqn: decomposition_weinstein_handles}
		X = (X_0, \Lambda_1, \dotsc, \Lambda_k) \cup_{\Lambda_1 = \Lambda_\varnothing} (B^{2n}_{std}, \Lambda_\varnothing)  \cup \dotsb \cup_{\Lambda_k = \Lambda_\varnothing }(B^{2n}_{std}, \Lambda_\varnothing)
		\end{equation}
		where the $i$th copy of $(B^{2n}_{std}, \Lambda_\varnothing)$ is glued to $X_0$ by identifying $\Lambda_\varnothing$ with $\Lambda_i$. Now we define $X_P$ to be the following Weinstein domain:
		\begin{equation}\label{eqn: decomposition_weinstein_P_handles}
		X_P:= (X_0, \Lambda_1, \dotsc, \Lambda_k) \cup_{\Lambda_1 = \Lambda_P} (B^{2n}_{std}, \Lambda_P) \cup \dotsb \cup_{\Lambda_k = \Lambda_P }(B^{2n}_{std}, \Lambda_P)
		\end{equation}
		Namely, we replace each standard Weinstein $n$-handle $(B^{2n}_{std}, \Lambda_\varnothing)$ by a $P$-Weinstein handle
		$(B^{2n}_{std}, \Lambda_P)$.
		
		\begin{remark}\label{rem: p-handles_isotropicsum}
			We note that attaching $P$-Weinstein handles $(B^{2n}_{std}, \Lambda_P)$ to $(X_0, \Lambda_1, \dotsc, \Lambda_k)$ is the same as attaching standard Weinstein handles $(B^{2n}_{std}, \Lambda_\varnothing)$ to $X_0$ with some modified attaching Legendrian $\Lambda_i^P \subset \partial X_0$. In fact, $\Lambda_i^P$  is the isotropic connected sum  $\Lambda_i \natural \Lambda_P$ of $\Lambda_i \subset \partial X_0$ and $\Lambda_P \subset \partial B^{2n}_{std}$, which we place into a Darboux chart in $\partial X_0$ disjoint from $\Lambda_i$.
			To see this, note that gluing $(B^{2n}_{std}, \Lambda_P)$ to $(X_0, \Lambda_i)$ by identifying $\Lambda_P$ with $\Lambda_i \subset \partial X_0$ is the same as gluing a cylinder $T^*(S^{n-1} \times D^1)$ to $(X_0, \Lambda_i) \amalg (B^{2n}_{std}, \Lambda_P)$ by identifying $S^{n-1}\times 0$ with $\Lambda_i$ and $S^{n-1}\times 1$ with $\Lambda_P$. The cylinder can be decomposed into a standard Weinstein index $1$ handle and a standard Weinstein index $n$ handle.
			So we first do simultaneous index $1$ handle attachment to $(X_0, \Lambda_i)$ and $(B^{2n}_{std}, \Lambda_P)$, with attaching sphere a point in $\Lambda_i$ and a point in $\Lambda_P$, to produce $(X_0 \natural B^{2n}_{std},  \Lambda_i \natural \Lambda_P)$. 
			If we identify $X_0 \natural B^{2n}$ with $X_0$, then $\Lambda_P$ becomes a Legendrian in $\partial X_0$ (in a Darboux chart disjoint from $\Lambda_i$) and $\Lambda_i\natural \Lambda_P$ is precisely 
			the isotropic connected sum of  $\Lambda_i$ and $\Lambda_P$ in $\partial X_0$. 
			Then we attach the (standard) index $n$ Weinstein handle of the cylinder $T^*(S^{n-1}\times D^1)$ along $\Lambda_i \natural \Lambda_P$. Thus, the decomposition of $X_P$ in Equation \ref{eqn: decomposition_weinstein_P_handles} can alternatively be described as
			\begin{equation}\label{eqn: decomposition_weinstein_handles_isotropic_sum}
			(X_0, \Lambda_1\natural \Lambda_P, \dotsc,  \Lambda_k\natural \Lambda_P) \cup_{\Lambda_1\natural \Lambda_P = \Lambda_\varnothing} (B^{2n}_{std}, \Lambda_\varnothing) \cup \dotsb \cup_{\Lambda_k \natural \Lambda_P= \Lambda_\varnothing }(B^{2n}_{std}, \Lambda_\varnothing)
			\end{equation}
			In particular, the attaching spheres for the (standard) index $n$ handles for $X$ and $X_P$ differ by a purely local modification, namely an isotropic connected sum with $\Lambda_P$.
		\end{remark}
		
		Now Claims 1), 2), 3) in Theorem \ref{thm: exotic_subdomains} follow from the analogous claims in Theorem \ref{thm: p-handles}. 
		For example, $X_\varnothing = X$ since $(B^{2n}_{std}, \Lambda_{\varnothing})$ is the standard Weinstein handle $(T^*D^n, \partial D^n)$. Also,  since $(B^{2n}_{std}, \Lambda_P)$ is a Weinstein subdomain of $(B^{2n}_{std}, \Lambda_Q)$ for $Q \subset P$,  $X_P$ is a Weinstein subdomain of $X_Q$ and this Weinstein embedding is also functorial with respect to inclusions of various subsets of primes. 
		If $0 \in P$, then $X_P$ is flexible. To see this, recall that $\Lambda_P \subset \partial B^{2n}_{std}$ is loose by Theorem \ref{thm: p-handles}; this implies that the attaching spheres $\Lambda_i^P \subset \partial X_0$ for $X_P$ are also loose since by Remark \ref{rem: p-handles_isotropicsum}, $\Lambda_i^P$ is the isotropic connected sum of $\Lambda_i$ with $\Lambda_P$, which is a loose Legendrian loosely embedded in a Darboux chart disjoint from $\Lambda_i$. If $0 \in Q \subset P$, then the cobordism between $X_P$ and $X_Q$ is flexible since the cobordism between $(B^{2n}_{std}, \Lambda_P)$ and $(B^{2n}_{std}, \Lambda_Q)$ is also flexible (in the complement of $\Lambda_P$).

		Finally, we compute $\Tw \mathcal{W}(X_P; \mathbb{Z})$. Since $X_P$ is a Weinstein subdomain of $X$, there is a Viterbo transfer functor:
		$$
		\Tw \mathcal{W}(X; \mathbb{Z}) \rightarrow \Tw \mathcal{W}(X_P; \mathbb{Z})
		$$
		As in the proof of Theorem \ref{thm: p-handles}, this functor is localization by $D_p \subset (T^*D^n, \partial D^n)$ (or equivalently by $D_p \natural \overline{D_p}$) and 
		$D_p \cong \mathrm{cone}(p \cdot \mathrm{Id}_{T^*_0 D^n})$. 
		On the other hand,  $T^*_0 D^n \subset (T^*D^n, \partial D^n) = (B^{2n}_{std}, \Lambda_\varnothing)$ is precisely the co-core $C_i^n$ of $H^n_{\Lambda_i}$ under the decomposition of $X$ in Equation \ref{eqn: decomposition_weinstein_handles} and so $D_p$ is isomorphic to $\mathrm{cone}(p \cdot \mathrm{Id}_{C_i^n})$.  
		By \cite{ganatra_generation, chantraine_cocores_generate}, the co-cores $C_i^n$ of all the $H^n_{\Lambda_i}$ generate $\Tw \mathcal{W}(X)$. So localizing by $\mathrm{cone}(p \cdot \mathrm{Id}_{C_i^n})$ for all $i$ is the same as localizing by 
		$\mathrm{cone}(p \cdot \mathrm{Id}_{L})$ for all $L \in \Tw  \mathcal{W}(X; \mathbb{Z})$. That is, 
		$\Tw \mathcal{W}(X_P; \mathbb{Z})\cong 
		\Tw \mathcal{W}(X; \mathbb{Z})[\frac{1}{P}]$ as desired. 
	\end{proof}
	
	We observe that our construction of $X_P$ depends on many choices. For example, it depends on the choice of initial Weinstein presentation for $X$. There are Weinstein homotopic presentations for $X$ with different numbers of index $n$ handles; hence in this case, our construction would involve carving out different numbers of Lagrangian disks (and then attaching the appropriate flexible cobordism). There are also choices to be made in constructing 
	the $P$-handles $(B^{2n}_{std}, \Lambda_P)$.  We fixed a $p$-Moore space $U \subset S^{n-1}$ so that $\tilde{C}^*(U) = 	\mathbb{Z}[-2] \overset{p}{\rightarrow} \mathbb{Z}[-3]$
	and used this to construct $D_p: = D_U$ and then form $(B^{2n}_{std}, \Lambda_P)$. In fact, we could have taken any $U \subset S^{n-1}$ so that $\tilde{C}^*(U)$ is quasi-isomorphic to  $\bigoplus_i (\mathbb{Z}[k_i+1] \overset{p}{\rightarrow} \mathbb{Z}[k_i])$ for any $k_i$. Repeating the construction for such $U$, we would also have 
	$\Tw \mathcal{W}(B^{2n}_{std}, \Lambda_P; \mathbb{Z}) \cong \Tw \mathcal{W}(B^{2n}_{std}, \Lambda_\varnothing; \mathbb{Z})[\frac{1}{P}]$ as well. 
	
	Now that we have described the subdomains $X_P$ of $X$, we can explain the difference between our construction and that of Abouzaid and Seidel \cite{Abouzaid_Seidel} more precisely. Abouzaid and Seidel \cite{Abouzaid_Seidel} starts with a Lefschetz fibration for $X^{2n}$ whose fiber is a Weinstein domain $F^{2n-2}$. They then embed the Lagrangian disks $D_p^{n-1}$ into $F^{2n-2}$ so that they are 	in a neighborhood of the co-cores $C_i^{n-1}$ of the critical index $n-1$ handles $H^{n-1}_i$ of $F^{2n-2}$;  using these disks, they build a larger fiber $F'$ (which has $F$ as a Weinstein subdomain) and add new vanishing cycles to create a new Lefschetz fibration, which is their space $X_P'$. On the other hand, the construction in Theorem \ref{thm: exotic_subdomains} embeds the disks $D_p^n$ into the \textit{total space} $X^{2n}$ so that they are in a neighborhood of the co-cores $C_i^n$ of the critical index $n$ handles $H^n_i$ of  $X^{2n}$; we then carve out these disks. The construction of Abouzaid-Seidel holds only for 
	$n \ge 6$. Because we work near the 	index $n$ handles instead of the index $n-1$ handles,  our construction improves this to hold for $n \ge 5$.

	Next we complete the proof of Corollary \ref{cor: exotic_subdomains_cotangent} concerning subdomains of $T^*M_{std}$.
	\begin{proof}[Proof of Corollary \ref{cor: exotic_subdomains_cotangent}]
		The only extra feature of this result over Theorem \ref{thm: exotic_subdomains} is the `only if' part of the statement:
		$T^*M_P \subset T^*M_Q$ if and only if $Q \subset P$ or $0 \in P$. 
		To prove this, we repeat the proof in Theorem \ref{thm: p-handles} 
		that $(B^{2n}_{std}, \Lambda_P)$ is a subdomain of $(B^{2n}_{std}, \Lambda_Q)$ if and only if $Q \subset P$.
		Namely, suppose that $T^*M_P \subset T^*M_Q$ is a Weinstein subdomain but $Q\not \subset P$ and $0 \not \in P$. Then there is a Viterbo localization functor on Fukaya categories over $\mathbb{F}_q$ for $q \in Q \backslash P$. However,
		$\Tw \mathcal{W}(T^*S^n_Q; \mathbb{F}_q) \cong 0$ but $\Tw \mathcal{W}(T^*S^n_P; \mathbb{F}_q) \cong \Tw \mathcal{W}(T^*S^n; \mathbb{F}_q) \cong \Tw C_*(\Omega S^n; \mathbb{F}_q)$ is non-trivial and so there cannot be such a localization functor. 
	\end{proof}
	\begin{remark}
		A similar argument using the fact that the Viterbo map on symplectic cohomology is a unital ring map shows that that $T^*S^n_P$ cannot be a \textit{Liouville} subdomain of $T^*S^n_Q$ if $Q \not \subset P$ and $0 \not \in P$.
	\end{remark}
	
	\subsubsection{Exotic presentations}
	We now briefly explain the connection between the subdomains of $T^*S^n_{std}$ constructed in Corollary \ref{cor: exotic_subdomains_cotangent} and certain `exotic' Weinstein presentations 
	of $T^*S^n_{std}$ studied by the first author in \cite{Lazarev_geometric_presentations}; the reader can safely skip this section without interrupting the flow of this paper.

There are many different Legendrian spheres $\Lambda_k \subset \partial B^{2n}_{std}$ so that $B^{2n}_{std}\cup H^n_{\Lambda_k}$ is Weinstein homotopic to the standard presentation $B^{2n}_{std}\cup H^n_{\Lambda_\varnothing}$; we call this an exotic presentation since $\Lambda_k$ is different from $\Lambda_\varnothing$. Under the resulting identification, the co-core of $H^n_{\Lambda_k}$ is $\natural^k_{i=1} T^*_{x_i} S^n \natural^{k-1}_{j=1} \overline{T^*_{y_j} S^n}$, the boundary connected sum of several copies of the cotangent fiber $T^*_q S^n$, possibly with the opposite orientation. 
	
Recall that for any $U \subset S^{n-1}$, we consider a Lagrangian disk $D_U \subset (T^*D^n, \partial D^n)$ with $D_U \cong \tilde{C}^{*-1}(U) \otimes T^*_0 D^n$ and  in the proof of Theorem \ref{thm: p-handles}, we observed that $T^*D^n\backslash D_U$ is the subcritical domain $T^*(S^{n-1}\times D^1)$. There are two disjoint Legendrian spheres $\Lambda_{1,U}, \Lambda_{2,U} \subset \partial(T^*(S^{n-1}\times D^1))$; here $\Lambda_{1,U}$ is the image of the original Legendrian stop $\partial D^n \subset \partial T^*D^n$ and $\Lambda_{2,U}$ is the Legendrian obtained by carving out $D_U$, i.e. the co-core of a handle attached along $\Lambda_{2,U}$ is $D_U$ (this Legendrian is called $\Lambda$ in the proof of Theorem \ref{thm: p-handles}). Starting from the standard presentation $ (B^{2n}_{std}, \Lambda_\varnothing) \cup_{\Lambda_\varnothing = \Lambda_\varnothing}  (B^{2n}_{std}, \Lambda_\varnothing)$ of $T^*S^n_{std}$ and taking $U$ to be a $p$-Moore space, the construction in Theorem \ref{thm: exotic_subdomains} produces the Weinstein subdomain $T^*S^n_p \subset T^*S^n_{std}$ as
	$$
	T^*S^n_p := (B^{2n}_{std}, \Lambda_\varnothing) \cup_{\Lambda_\varnothing  = \Lambda_{1,U}} (T^*(S^{n-1}\times D^1)), \Lambda_{1,U}, \Lambda_{2,U, loose}) \cup_{\Lambda_{2,U,loose} = \Lambda_\varnothing} (B^{2n}_{std}, \Lambda_\varnothing)
	$$
where $\Lambda_{2,U,loose}$ is a loose version of $\Lambda_{2,U}$. That is, $T^*S^n_p$ is obtained by attaching a flexible handle to the (unstopped) Weinstein domain 
	$$
	(B^{2n}_{std}, \Lambda_\varnothing) \cup_{\Lambda_\varnothing  = \Lambda_{1,U}} (T^*(S^{n-1}\times D^1)), \Lambda_{1,U})
	$$
where $U$ is a $p$-Moore space. 
	
Now, if $U$ is a neighborhood of $\coprod^{k}_{i=1} B^{n-1} \amalg \vee^{k-1}_{j=1} S^1 \subset S^{n-1}$, the disjoint union of $k$ balls $B^{n-1}$ and the wedge sum of $k-1$ copies of $S^1$, then $D_U$ is Lagrangian isotopic in $(T^*D^n, \partial D^n)$ to $\natural^k_{i=1} T^*_{x_i} D^n \natural^{k-1}_{j=1} \overline{T^*_{y_j} D^n}$. Furthermore, $\Lambda_{1,U} \subset \partial (T^*S^{n-1}\times D^1)$ is loose. This is because the subdomain obtained by carving out the disjoint union $\coprod^k_{i=1} T^*_{x_i} D^n \coprod^{k-1}_{j=1} \overline{T^*_{y_j} D^n}$ and the subdomain obtained by carving out the boundary connected sum $\natural^k_{i=1} T^*_{x_i} D^n \natural^{k-1}_{j=1} \overline{T^*_{y_j} D^n}$ are related by flexible cobordism (see  \cite{Lazarev_geometric_presentations} for the proof). In particular, the attaching spheres for the cobordism are loose in the complement of the stop. Since the former domain has a loose stop, so does the latter by \cite{CE12}.
Hence
$$
	(B^{2n}_{std}, \Lambda_\varnothing) \cup_{\Lambda_\varnothing  = \Lambda_{1,U}} (T^*(S^{n-1}\times D^1)), \Lambda_{1,U})
	$$
is a flexible Weinstein domain $X$, with no stop. In fact, $X$ is the standard Weinstein ball because 
$
[\natural^k_{i=1} T^*_{x_i} D^n \natural^{k-1}_{j=1} \overline{T^*_{y_j} D^n}] = [T^*_x D^n] \in H^n(T^*D^n; \mathbb{Z})$, which implies that it has trivial homology 
and hence is a smooth ball by the h-cobordism theorem. 
In conclusion,  the stopped domain
	$$
	(B^{2n}_{std}, \Lambda_\varnothing) \cup_{\Lambda_\varnothing  = \Lambda_{1,U}} (T^*(S^{n-1}\times D^1)), \Lambda_{1,U},  \Lambda_{2,U})
	$$ 
is precisely $(B^{2n}_{std}, \Lambda_k)$, where $\Lambda_k$ is the Legendrian from \cite{Lazarev_geometric_presentations}, since by construction the co-core of a handle attached along $\Lambda_{2,U}$ is $D_U= \natural^k_{i=1} T^*_{x_i} D^n \natural^{k-1}_{j=1} \overline{T^*_{y_j} D^n}$.

We end with a discussion of which ingredients were necessessary in the construction of these exotic presentations. First and foremost, we need to realize $ \natural_{i=1}^k T^*_{x_i} D^n \natural_{j=1}^{k-1} \overline{T^*_{y_j} D^n}$ as $D_U$ and hence embed  $S^1$ into $S^{n-1}$ as a proper subset. This requires $n \ge 3$ and indeed \cite{Lazarev_geometric_presentations} proves there are no such exotic presentations for $n = 2$. Interestingly, these exotic presentations fail to exist for the same reason that the existence h-principle for Legendrians fails when $n = 2$. Indeed, consider a proper subdomain $U \subset B^{n-1} \subset S^{n-1}$ and a Legendrian $\Lambda^{n-1} \subset \partial B^{2n}_{std}$. Then one can form the Legendrian $U$-stabilization $St_U(\Lambda) \subset \partial B^{2n}_{std}$ of  $\Lambda$ (see \cite{CE12, Murphy11}) whose Thurston-Bennequin invariant differs from that of $\Lambda$ by the Euler characteristic $\chi(U)$. If $n \ge 3$, any integer can be realized as the Euler characteristic of $U$ by taking $U \subset B^{n-1}$ to be a neighborhood of $\coprod^j B^{n-1} \amalg \vee^k S^1 \subset S^{n-1}$, the space we previously considered. This allows one to realize all formal Legendrian embeddings by actual Legendrians for $n \ge 3$. However, if $n = 2$, all proper subdomains $U \subset S^1$ have $\chi(U) > 0$. Indeed the Bennequin inequality proves that the Thurston-Bennequin invariant of any smoothly trivial Legendrian in $\partial B^{4}_{std}$ is at most $-1$; so the existence h-principle fails for $n =2$. We expect that there is precise connection between the $U$-stabilized Legendrians $St_U(\Lambda)$ from \cite{CE12, Murphy11} and the Lagrangian disks $D_U$ here, since both construction involving pushing a smooth subdomain $U$ in one Legendrian through another.
	
\section{Classifying Lagrangian disks}\label{sec:class_lag_disks}

In this section, we prove Theorem \ref{thm: Lagrangian_homology_twisted}: if $M$ is simply connected and spin, and $i\colon L \hookrightarrow T^*M$ is null-homotopic, then $L \cong CW^*(M, L) \otimes T^*_q M^n$ in $\mathrm{Tw}\, \mathcal{W}(T^*M; \mathbb{Z})$.
To accomplish this, we will apply Koszul duality to characterize objects of $\mathrm{Tw}\,\mathcal W(T^*M;\mathbb Z)$ as modules over the $A_\infty$-algebra $CW^*(M,M)\cong C^*(M)$.
Here it is  crucial that we work with the $\mathbb{Z}$-graded wrapped Fukaya category, where the $\mathbb{Z}$-grading comes from  the Lagrangian fibration by cotangent fibers. Any Lagrangian disk, since it is contractible, can be $\mathbb{Z}$-graded; the zero-section $M \subset T^*M_{std}$ can also be $\mathbb{Z}$-graded for
this grading. Hence these Lagrangians define objects of the $\mathbb{Z}$-graded Fukaya category.

\subsection{$C^*(X)$-modules}\label{sec:C*_mod}
We begin with a general discussion of how to view Floer complexes as modules over Morse cochain algebras. The outcome is Proposition \ref{prop: C* modules homotopic}, which says that the module structures are unexpectedly topological. This is what will allow us to draw Floer theoretic conclusions from the topological assumption of null-homotopy.

For now, we will work in a general Liouville domain $X$. Given two Lagrangian branes $K,L\subset X$, possibly equipped with rank $1$ local systems, we can endow $CW^*(K,L)$ with the structure of a right $C^*(X)$-module in a number of ways. In each case, we model the $A_\infty$ structure on our cochain algebras $C^*(X)$, $C^*(K)$, and $C^*(L)$ with Morse complexes and perturbed gradient flow trees \cite{Abouzaid2011a} associated to exhausting Morse functions $f_X$, $f_K$, and $f_L$.

Let us fix some notation. The moduli space of domains controlling the $A_\infty$ operations is the space
\[
\begin{tikzcd}\mathcal{T}^{d+1}\ar[d]\\\mathcal{R}^{d+1}\end{tikzcd}
\]
of metric ribbon trees with $d+1$ infinite leaves and no finite leaves, labeled $x_0,\dotsc,x_d$ in counterclockwise order. More explicitly, a point $p\in\mathcal R^{d+1}$ is an isomorphism class $[T_p]$, where $T_p$ is a noncompact tree with
\begin{itemize}
	\item $d+1$ ends and no mono- or bivalent vertices,
	\item a ribbon structure, which for a tree is the same as a homotopy class of planar embeddings,
	\item an edge metric, meaning that we can measure the distance between any two points of $T_p$ (not necessarily vertices), and
	\item a labeling of the ends by $x_0,\dotsc,x_d$ in counterclockwise order with respect to the ribbon structure.
\end{itemize}
The fibration $\mathcal T^{d+1}\to\mathcal R^{d+1}$ is the tautological one, which over each $p$ is a representative $T_p$. In what follows, we will imagine $x_0$ as the bottom of $T_p$ and the other $x_i$ as the top, which will allow us to use the prepositions ``below'' or ``above'' to mean ``closer to $x_0$'' or ``closer to some other $x_i$'', respectively.

The restriction homomorphisms $i_K^* \colon C^*(X)\to C^*(K)$ and $i_L^* \colon C^*(X)\to C^*(L)$ are controlled by the space
\[
\begin{tikzcd}\mathcal{G}^{d+1}\ar[d]\\\mathcal{S}^{d+1}\end{tikzcd}
\]
of \emph{grafted trees}, which are metric ribbon trees $T$ as above with the additional data of a (necessarily finite) subset $D\subset T$ which separates $x_0$ from the other leaves and whose elements are equidistant from $x_0$. For $d\ge2$, $\mathcal S^{d+1}$ has a natural $\mathbb R$-action which translates $D$, and the quotient is canonically identified with $\mathcal R^{d+1}$ (for $d=1$, $\mathcal S^{1+1}$ is a single point). However, the natural compactification $\overline{\mathcal R}^{d+1}$ models the associahedron, while $\overline{\mathcal S}^{d+1}$ models the multiplihedron. The restriction homomorphism
\[
\{F^d\,\vert\,d=1,\dotsc,\infty\}\colon C^*(X)\to[C^*(K)\text{ or }C^*(L)]
\]
is then given by counting isolated perturbed gradient flow trees of shape $T_q$ for some $q\in\mathcal S^{d+1}$, where the portion of $T_q$ above (resp. below) $D$ maps into $X$ (resp. $K$ or $L$). Note that, 
because we work with a \emph{perturbed} gradient flow, we do not need to require $f_X$ to restrict to $f_K$ or $f_L$. Of course, if we wanted to we could arrange that $f_X$ restricts to one of these Morse functions, but generally it would impossible to achieve both. Fortunately, all the resulting homomorphisms are homotopic.

To make Floer complexes into $C^*(X)$-modules, we need chain-level PSS-type structures, which are built from \emph{short trees} or \emph{short grafted trees}. A short tree with $d$ inputs is a rooted metric ribbon tree with $d$ infinite leaves and no finite leaves (except possibly the root). The root is labeled $y$, while the leaves are labeled $x_1,\dotsc,x_d$ in counterclockwise order. A short grafted tree is a short tree equipped with the additional data of a dividing set $D$ as above either separating $y$ from the $x_i$ or equal to $\{y\}$. We will denote the spaces of short trees and short grafted trees by $\mathcal R^{d+1}_s$ and $\mathcal S^{d+1}_s$, respectively. There are canonical piecewise smooth homeomorphisms
\begin{equation}\label{eq:short_tree_descr}
\mathcal R^{d+1}_s\cong\mathcal R^{d+1}\times\R_{\ge0}
\end{equation}
for $d\ge2$
and
\begin{equation}\label{eq:short_grafted_descr}
\mathcal S^{d+1}_s\cong\mathcal R^{d+1}_s\times\R_{\ge0}
\end{equation}
for $d\ge1$, i.e. all $d$. In \eqref{eq:short_tree_descr}, the $\mathbb R_{\ge0}$ factor measures the distance between the root $y$ and the first vertex, while in \eqref{eq:short_grafted_descr} it measures the distance between $y$ and the dividing set.

The PSS-type structures in question all come from moduli spaces of strips with some number of short Morse trees attached at marked points.

\begin{definition}\label{def: hedge}
	A \emph{hedge} comprises
	\begin{enumerate}
		\item a smooth function $f\colon\R\to[0,1]$,
		\item a collection of $k$ points $z_1,\dotsc,z_k$ on the graph $\Gamma(f)\subset\R\times[0,1]$ with strictly increasing $\R$ components, and
		\item for each $z_i$, a short tree $T_i$.
	\end{enumerate}
	Identifying $z_i$ with the root $y_i$ of the tree $T_i$ induces a total lexicographic order of the leaves $x_{ij}$ of the trees $T_i$, namely $x_{i,j}<x_{i'j'}$ if either $i<i'$ or both $i=i'$ and $j<j'$.
\end{definition}

Fix a number $c\in(0,1)$. The space $\mathcal H^d_c$ of hedges with $d$ leaves $x_{ij}$ and $f(s)=c$ comes a priori as a disjoint union of components indexed by partitions of the leaves into trees $T_i$. However, there is a natural way to glue the various components to build a connected moduli space. To see this, note that the boundary strata (before compactification) come from one or more roots $y_i$ becoming multivalent, or in horticultural terms from some tree $T_i$ becoming maximally short. Such configurations can also be achieved by having multiple smaller short trees attached to distinct marked points collide. The result is that we can make $\mathcal H^d_c$ into a \emph{connected, smoothly stratified, topological manifold without boundary}, see Figure \ref{fig: moduli}. This is good enough to construct operations in Floer theory.

$\mathcal H^d_c$ has a natural compactification $\overline{\mathcal H}^d$, where the codimension $1$ boundary strata come in two types. The first is associated with Morse breaking, where a single short tree will break into a short tree and a (long) tree. The second is a type of Floer breaking associated with the marked points $z_i$ moving apart, so that the limiting configuration is made up of two hedges.

	\begin{figure}
		\centering
		\includegraphics[scale=0.15]{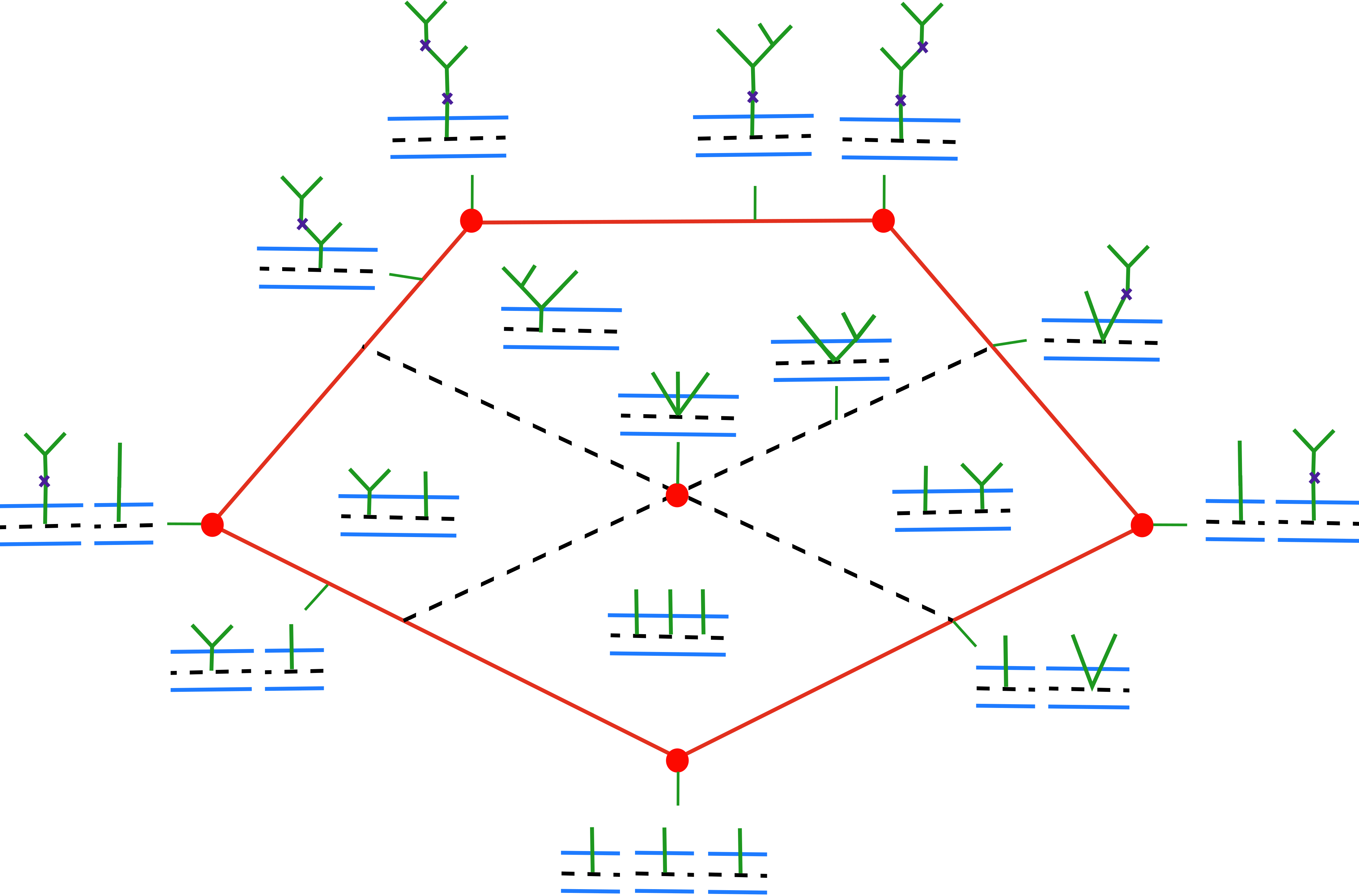}
		\caption{The space $\overline{\mathcal H}^3_c$; the boundary consists of Floer breaking (lower boundary in the diagram) and Morse breaking (upper boundary).
		}
		\label{fig: moduli}
	\end{figure}

An $X$-valued perturbation datum for a hedge $\mathbf H$ amounts to a perturbation datum for each short tree $T_i$, which is just an $\epsilon$-parametrized family of vector fields on $X$ for each edge $\epsilon$ of $T_i$ which vanishes outside a compact subset of $\epsilon$. Given a Morse-Smale pair on $X$ and a Floer datum for the pair $(K,L)$, we can define a hedge map out of $\mathbf H$ to be a tuple $(u,\tau_1,\dotsc,\tau_k)$, where
\begin{itemize}
	\item $u$ is a Floer trajectory with boundary on $(K,L)$,
	\item $\tau_i$ is a perturbed gradient flow tree in $X$ parametrized by $T_i$, and
	\item $\tau_i(y_i)=u(z_i)$.
\end{itemize}
If $\mathbf H\in\mathcal H^d_0$, we can analogously define a $K$-valued perturbation datum for $\mathbf H$ to be a family of vector fields on $K$, and a hedge map to involve gradient flow trees in $K$; if $\mathbf H\in\mathcal H^d_1$, we can do the same with $L$.

For generic Morse-Smale pairs, smooth translation-invariant families of perturbation data on $\mathcal H^d_c$, and Floer data on $X$, the spaces of $d$-leaved hedge maps are smoothly stratified topological manifolds of the expected dimension. Counting such maps which are isolated up to translation makes $CW^*(K,L)$ into a right $C^*(X)$-module, which we'll denote $CW^*(K,L)_{X,c}$. Similarly, when $c=0$ or $1$, we can make $CW^*(K,L)$ into a \emph{right} $C^*(K)$- or $C^*(L)$-module $CW^*(K,L)_{K,0}$ or $CW^*(K,L)_{L,1}$, respectively.

The key holomorphic curve ingredient of our story is that these modules are all homotopic (and therefore quasi-isomorphic) when pulled back to $C^*(X)$:

\begin{proposition}\label{prop: C* modules homotopic}
	For $c_1, c_2\in[0,1]$, there is a homotopy
	\begin{equation}\label{eqn: module htopy easy}
	CW^*(K,L)_{X,c_1} \simeq CW^*(K,L)_{X,c_2}
	\end{equation}
	of right $C^*(X;\mathbb Z)$-modules.
	
	Similarly, for any $c$, there are homotopies
	\begin{align}
	CW^*(K,L)_{X,c} &\simeq i_K^\vee CW^*(K,L)_{K,0}
	\label{eqn: module htopies 1}\\
	CW^*(K,L)_{X,c} &\simeq i_L^\vee CW^*(K,L)_{L,1}
	\label{eqn: module htopies 2},
	\end{align}
	where
	\[
	i_K^\vee \colon \mathrm{Mod}_{C^*(K;\mathbb Z)} \to \mathrm{Mod}_{C^*(X;\mathbb Z)}
	\]
    is the pullback functor under the restriction homomorphism of cochains 
	$i_K^*: C^*(X) \rightarrow C^*(K)$, and similarly for $i_L^\vee$.
\end{proposition}
\begin{remark}
	The key takeaway of Proposition \ref{prop: C* modules homotopic} is not just that $CW^*(K,L)$ has a canonically defined $C^*(X)$-module structure, but that this module structure is \emph{determined by} either the $C^*(K)$- or the $C^*(L)$-module structure.
\end{remark}
\begin{proof}
	Pick a smooth function $f\colon\R\to[0,1]$ interpolating between $f(s) = c_1$ for $s$ near $+\infty$ and $f(s)=c_2$ for $s$ near $-\infty$. Write $\mathcal H^d_f$ for the space of hedges with marked points $z_i$ on the graph of $f$. Counting isolated (no longer up to translation) hedge maps parametrized by $\mathcal H^d_f$ defines the homotopy \eqref{eqn: module htopy easy}.
	
	For the second part, we prove \eqref{eqn: module htopies 1}, since the proof of \eqref{eqn: module htopies 2} is identical. For this, we may apply the first part to assume $c=0$, so we need only produce a homotopy between $CW^*(K,L)_{X,0}$ and $CW^*(K,L)_{K,0}$. We do this by generalizing the notion of a hedge to that of a \emph{grafted hedge}. This is the same as Definition \ref{def: hedge}, except $f\equiv0$ and the short trees $T_i$ are replaced by short grafted trees. A (ordinary) hedge can thus be viewed as a special case of a grafted hedge, where all the dividing points are at the root $y_i$. Using this identification, we can extend the definition of the spaces $\mathcal H^d_c$ to negative values of $c$. Concretely, we declare $\mathcal H^d_c$ to be the space of $d$-leaved grafted hedges, where each tree is attached to the strip at $t=0$ and has dividing set at distance $\lvert c\rvert$ from the root. For negative $c$, $\mathcal H^d_c$ continues to have a natural compactification $\overline{\mathcal H}^d_c$, and there is a canonical diffeomorphism $\overline{\mathcal H}^d_c\cong\overline{\mathcal H}^d_{c'}$ for any $c,c'\in(-\infty,1]$.
	
	For $\mathbf H$ a grafted hedge, a hedge map out of $\mathbf H$ is a tuple $(u,\tau_1,\dotsc,\tau_k)$, where
	\begin{itemize}
		\item $u$ is a Floer trajectory with boundary on $(K,L)$.
		\item $\tau_i$ is a perturbed grafted gradient flow tree with leaves in $X$ and root in $K$ parametrized by $T_i$, and
		\item $\tau_i(y_i)=u(z_i)$.
	\end{itemize}
	Now the diffeomorphism $\overline{\mathcal H}^d_c\cong\overline{\mathcal H}^d_{c'}$ is compatible with both the internal stratification and the boundary decompositions, so follows that hedge maps parametrized by $\mathcal H^d_c$ continue to define $C^*(X)$-module structures $CW^*(K,L)_{X,c}$ for $c<0$. Moreover, the same argument as for nonnegative $c$ shows that these module structures are homotopic -- just interpolate the dividing sets rather than the attaching points.
	
	To conclude, observe that the pullback module $i_K^*CW^*(K,L)_{K,0}$ is what we get by sending the dividing set to infinity. While it is delicate to do that directly, it is enough to move the dividing set close to infinity: below any given action bound, gluing theory establishes a bijection of spaces of hedge maps. This ensures first that the module structure maps stabilize to the pulled-back ones, and second that the homotopies eventually become trivial.
\end{proof}

\begin{remark}\label{rem: Hochschild cohomology}
A version of proposition \ref{prop: C* modules homotopic} remains true with $C^*(X)$ replaced by symplectic cochains $SC^*(X)$, $C^*(K\text{ or }L)$ replaced by $CW^*(K\text{ or }L)$, and the restriction maps replaced by closed-open maps. In that case, one is forced to use \emph{left} $CW^*(L)$-modules. While we expect all of the resulting homotopies to be intertwined by the relevant $A_\infty$ algebra homomorphisms, sticking to Morse cochains allows us to avoid a good deal of combinatorial messiness.
	
Recall that for a Weinstein domain $X$, $SC^*(X)$ is quasi-isomorphic to the Hochschild cochains  $CC^*(\mathcal{W}(X))$ of $\mathcal{W}(X)$ \cite{Ganatra_thesis}. Using this quasi-isomorphism, we note that
Proposition \ref{prop: C* modules homotopic} has a purely categorical analog. For any $A_\infty$ category $\mathcal{A}$,  there is an $A_\infty$-homomorphism $CC^*(\mathcal{A}) \rightarrow \hom^*(X, X)$ and hence a pullback map on modules, i.e. 
\[
\pi_X: \mathrm{Mod}_{\mathrm{end}^*(X)} \rightarrow \mathrm{Mod}_{CC^*(\mathcal{A})}.
\]
Since $CC^*(\mathcal{A})$ is an $E_2$-algebra, there is also an $A_\infty$-homomorphism $CC^*(\mathcal{A}) \rightarrow \hom^*(X, X)^{op}$ and hence a similar pullback functor
\[
\overline{\pi}_X: \mathrm{Mod}_{\mathrm{end}^*(X)^{op}} \rightarrow \mathrm{Mod}_{CC^*(\mathcal{A})}.
\]
For any two objects $X, Y\in\mathcal{A}$, composition of morphisms in $\mathcal{A}$ makes  $\hom(X,Y)$ an object of $\mathrm{Mod}_{\mathrm{end}(X)}$ and also of $\mathrm{Mod}_{\mathrm{end}(Y)^{op}}$. Then the categorical analog of Proposition \ref{prop: C* modules homotopic} is that the objects $\pi_X \hom(X,Y)$ and $\overline{\pi}_Y \hom(X,Y)$ are quasi-isomorphic in $\mathrm{Mod}_{CC^*(\mathcal{A})}$.

For the actual statement in Proposition \ref{prop: C* modules homotopic}, we work with $C^*(X)$, the low-action part of $CC^*(\mathcal{W}(X))$, and need to identify $CC^*(\mathcal{W}(X)) \rightarrow \hom^*(L,L)$ with the restriction map $C^*(X) \rightarrow C^*(L)$ on Morse cochains. 
Here it is essential that our Lagrangian $L$ \textit{is not} equipped with a bounding cochain, which destroys the action filtration on Floer cochains and hence our access to the low-energy, topological subcomplex.
	\end{remark}

While so far we have considered general $A_\infty$ presentations of our Morse cochain complexes, the above constructions work just as well for their strict unitalizations $C^*_s(-)$. Indeed, suppose $X$ is connected, and pick a positive exhausting Morse function $f$ on $X$ with a unique degree $0$ critical point. Define
\[
C^*_s(X) := CM^{\ge1}(f)\oplus\mathbb Z\cdot\mathbf 1
\]
with the restricted $A_\infty$ structure on $CM^{\ge1}(f)$ (which is well-defined because $\mu^k$ increases reduced degree, which is non-negative by assumption) 
and for which $\mathbf 1$ is a strict unit. Any $A_\infty$ homomorphism
\[
C^*(X)\to\mathcal A
\]
for $\mathcal A$ a strictly unital $A_\infty$ algebra induces a strictly unital homomorphism
\[
C^*_s(X)\to\mathcal A.
\]
Because modules are just functors to the strictly unital dg-category $\mathbf{Ch}$, we conclude

\begin{corollary}\label{cor: strictly unital modules homotopic}
	If $X$, $K$, and $L$ are connected, then Proposition \ref{prop: C* modules homotopic} continues to hold in the realm of strictly unital modules with $C^*(X)$ replaced with $C^*_s(X)$, and similarly with $K$ and $L$.
	\qed
\end{corollary}

\begin{corollary}\label{cor: factoring trough the aug}
Let $M$ be a closed connected manifold. If the restriction $A_\infty$-homomorphism $i^*\colon C^*(T^*M;\mathbb Z)\to C^*(L)$ factors up to homotopy through the canonical augmentation,
\[
\begin{tikzcd}
C^*(T^*M;\mathbb Z)\ar[rr, "i^*"]\ar[dr, "\varepsilon_{\mathrm{can}}"] && C^*(L;\mathbb Z)\\
&\mathbb Z\ar[ur, "\eta"]
\end{tikzcd}
\]
then $CW^*(M,L)_{M,0}$ is isomorphic to a module in the image of
\[
\Tw \mathbb Z \subset \mathrm{Mod}_\mathbb Z \xrightarrow{\varepsilon_{\mathrm{can}}}\mathrm{Mod}_{C^*(M;\mathbb Z)}.
\]
\end{corollary}

\begin{proof}
	Replacing $C^*(-)$ by $C^*_s(-)$, we may assume all algebras and maps are strictly unital. In particular, the pullback functor
	\[
	\eta^*\colon\mathrm{Mod}_{C^*(L;\mathbb Z)}\to\mathrm{Mod}_{\mathbb Z}
	\]
	preserves strict unitality of modules. Since a strictly unital $\mathbb Z$-module is just a chain complex, the $\mathbb Z$-module $\eta^*(CW(M,L)_{L,1})$ coincides with its underlying chain complex, which lies in $\Tw \mathbb Z$ because $M$ is compact.
	
	The result now follows from Corollary \ref{cor: strictly unital modules homotopic} (on each connected component of $L$), together with the observation that the restriction $C^*(T^*M)\to C^*(M)$ is an isomorphism.
\end{proof}

\subsection{Disks in cotangent bundles}

In the previous section, we studied properties of Floer modules $CW^*(K, L)$ over various Morse cochain algebras. 
In this section, we restrict to the case of $T^*M$, where $M$ is a simply connected, spin manifold. We use Koszul duality to show that the module structure over $C^*(M)$ knows everything about the Fukaya category and prove Theorem \ref{thm: Lagrangian_homology_twisted}.

We first construct a presentation of the wrapped Fukaya category which is well-adapted to talking about modules over $C^*(M)$. First, write $\mathcal C$ for the semiorthogonally glued category
\[
\langle M^\mathrm{Morse}, \mathcal W(T^*M)\rangle,
\]
where $\mathrm{end}^*(M^\mathrm{Morse})=C^*_s(M)$, and $\hom^*_{\mathcal C}(M^\mathrm{Morse}, L) = CW^*(M,L)$. The mixed $A_\infty$ operations count generalized hedges, i.e. usual perturbed holomorphic disks whose first boundary lies geometrically on $M$, together with short perturbed gradient flow trees in $M$ attached at boundary marked points. We will obtain our desired presentation by localizing $\mathcal C$:

\begin{lemma}\label{lem:Morse_wrapped_cat}
  Let $e\in\mathrm{hom}^0(M^\mathrm{Morse},M)$ be a cocycle representing the unit in $CW^*(M,M)$. Define
  \begin{equation}
    \mathcal W^\mathrm{Morse}(T^*M) := \mathcal C / \mathrm{cone}(e),
  \end{equation}
  so that we have tautological functors
  \begin{equation*}
  \begin{tikzcd}
    \mathcal W(T^*M) \ar[r,"i_\mathcal{W}" above] & \mathcal W^\mathrm{Morse}(T^*M)
      & \mathrm{end}^*_{\mathcal C}( M^\mathrm{Morse} ) = C^*_s( M ) \ar[l, "i_M" above].
  \end{tikzcd}
  \end{equation*}
  Then $i_\mathcal{W}$ is a quasi-equivalence and $i_M$ is fully faithful.
\end{lemma}
\begin{proof}
  For any object $X\in\mathcal W(T^*M)$, precomposition with $e$ induces a quasi-isomorphism
  \[
    \hom_\mathcal{C}^*( M, X ) \cong \hom_\mathcal{C}^*( M^\mathrm{Morse}, X ).
  \]
  This means that $\mathrm{cone}( e )$ is left-orthogonal to every $X \in \mathcal W( T^*M )$, which implies that $i_\mathcal W$ is fully faithful. Because $i_\mathcal W(M)$ is isomorphic to $M^\mathrm{Morse}$ in $\mathcal W^\mathrm{Morse}( T^*M )$, $i_\mathcal W$ is also essentially surjective, which means it's an equivalence.

  The proof for $i_M$ is identical, except $\mathrm{cone}(e)$ is right-orthogonal to $M^\mathrm{Morse}$ by the classical Lagrangian PSS isomorphism.
\end{proof}

The benefit of $\mathcal W^\mathrm{Morse}(T^*M)$ is that it allows for direct Koszul duality between the Morse cochain algebra on the zero section and the wrapped Fukaya algebra of the fiber.
In particular, we do not have to transfer Corollary \ref{cor: factoring trough the aug} through Floer's isomorphism.

\begin{proposition}\label{prop: full_faithful_zero_section}
  If $M$ is a simply connected, spin manifold, then the restricted Yoneda functor
  \begin{equation*}
  \begin{tikzcd}
    \mathcal Y \colon \mathcal W^\mathrm{Morse}( T^*M ) \ar[r, "\mathrm{Yoneda}"] &
    \mathrm{Mod}_{\mathcal W^\mathrm{Morse}( T^*M )} \ar[r, "i_M^*"] &
    \mathrm{Mod}_{C^*(M)}
  \end{tikzcd}
  \end{equation*}
  is fully faithful.
\end{proposition}
\begin{remark}
  Note that we have used the full-faithfullness of $i_M$ from Lemma \ref{lem:Morse_wrapped_cat} to write $C^*(M)$ rather than $\mathrm{end}^*(M^\mathrm{Morse})$.
  At the level of objects, $\mathcal Y$ just sends $L$ to $CW^*(M, L)_{M,0}$.
\end{remark}
\begin{proof}
  Lemma \ref{lem:Morse_wrapped_cat} and Abouzaid's theorems \cite{Abouzaid_based_loops, abouzaid_cotangent_fiber} give us a chain of quasi-equivalences
  \begin{equation*}
    \begin{tikzcd}[cramped, sep=small]
    \mathrm{Tw}\,\mathcal W^\mathrm{Morse}(T^*M) \ar[r, "\cong" above] \ar[rrr, bend right=10, "F"] &
    \mathrm{Tw}\,\mathcal W(T^*M) \ar[r, "\cong" above] &
    \mathrm{Tw}\,( \mathrm{end}^*(T_q^*M) ) \ar[r, "\cong" above] &
    \mathrm{Tw}\,( C_{-*}( \Omega M )).
  \end{tikzcd}
  \end{equation*}
  The resulting functor $F$ sends the cotangent fiber $T^*_qM$ to the rank $1$ free module.

  Let us study what happens to $M^\mathrm{Morse}$. We know $CW^*(M^\mathrm{Morse}, T_q^*M) \cong \mathbb Z$, since the zero section and fiber have just one intersection point. This means that $F(M^\mathrm{Morse})$ is an augmentation, and in fact it is the canonical augmentation of $C_{-*}(\Omega M)$. Indeed, all $C_{-*}(\Omega M)$-modules whose cohomology is $\mathbb{Z}$ are quasi-isomorphic. To see, use the homological perturbation lemma to replace $C_{-*}(\Omega M)$ with its cohomology $H_{-*}(\Omega M)$. This is supported in non-positive degrees and, because $M$ is simply connected, has $H_0(\Omega M) \cong \mathbb{Z}$. Since the $A_\infty$-module operation 
  \[\mu^{k \vert 1}\colon H_{-*}(\Omega M)^{\otimes k} \otimes \mathbb{Z} \rightarrow \mathbb{Z}\]
  has degree $1-k$ and $H_{-*}(\Omega M)$  is supported in non-positive degrees, the only non-trivial $A_\infty$-operation is the product $\mu^{1\vert1}: H_0(\Omega M)  \otimes \mathbb{Z} \rightarrow \mathbb{Z}$; this is the identity operation.

By \cite{Adams_cobar}, the standard augmentation and the rank 1 free module of $C_{-*}(\Omega M)$ are Koszul dual if $M$ is simply-connected, so $M^\mathrm{Morse}$ is Koszul dual to $T_q^*M$ and the proposition follows.
\end{proof}

\begin{remark}
	Simply-connectedness and $\mathbb Z$-grading are standard essential ingredients for Koszul duality. The spin condition also seems essential in our proof, but we do not have an example showing that Proposition \ref{prop: full_faithful_zero_section} fails without it. 
\end{remark}

We now have the necessary ingredients to prove Theorem \ref{thm: Lagrangian_homology_twisted}.

\begin{proof}[Proof of Theorem \ref{thm: Lagrangian_homology_twisted}]
 We now turn to the Lagrangian $L \subset T^*M$. The hypothesis that $L$ is null-homotopic in $T^*M$ implies that the hypothesis of Corollary \ref{cor: factoring trough the aug} is satisfied. This means that, up to isomorphism, $\mathcal Y( L )$ is the finite-dimensional cochain complex $CW^*(M,L)$, i.e. a complex of standard augmentations. On the other hand, the same reasoning (or a direct appeal to Corollary \ref{cor: strictly unital modules homotopic}) shows that $\mathcal Y( T_q^*M )$ is itself a standard augmentation and hence $\mathcal Y (CW^*(M, L) \otimes T^*_q M)$ is also the finite-dimensional cochain complex $CW^*(M,L)$. 
 Since $\mathcal Y$  is full and faithful by Proposition \ref{prop: full_faithful_zero_section}, $L$ is quasi-isomorphic to $CW^*(M, L) \otimes T^*_q M$ as desired.
 \end{proof}

	\bibliographystyle{abbrv}
	\bibliography{sources}

\begin{thebibliography}{10}

\bibitem{abouzaid_cotangent_fiber}
M.~Abouzaid.
\newblock A cotangent fibre generates the {F}ukaya category.
\newblock {\em Advances in Mathematics}, 228(2):894 -- 939, 2011.

\bibitem{Abouzaid2011a}
M.~Abouzaid.
\newblock A topological model for the {F}ukaya categories of plumbings.
\newblock {\em J. Differential Geom.}, 87(1):1--80, 2011.

\bibitem{Abouzaid}
M.~Abouzaid.
\newblock Nearby {L}agrangians with vanishing {M}aslov class are homotopy
  equivalent.
\newblock {\em Invent. Math.}, 189(2):251--313, 2012.

\bibitem{Abouzaid_based_loops}
M.~Abouzaid.
\newblock On the wrapped {F}ukaya category and based loops.
\newblock {\em J. Symplectic Geom.}, 10(1):27--79, 2012.

\bibitem{abouzaid_seidel_recombination}
M.~Abouzaid and P.~Seidel.
\newblock Altering symplectic manifolds by homologous recombination, 2010.
\newblock arXiv:1007.3281.

\bibitem{Abouzaid_Seidel}
M.~Abouzaid and P.~Seidel.
\newblock An open string analogue of {V}iterbo functoriality.
\newblock {\em Geom. Topol.}, 14(2):627--718, 2010.

\bibitem{Adams_cobar}
J.~Adams.
\newblock On the cobar construction.
\newblock {\em Proc. Nat. Acad. Sci.}, 42:409--412, 1956.

\bibitem{Auroux_Smith_surfaces}
D.~Auroux and I.~Smith.
\newblock Fukaya categories of surfaces, spherical objects, and mapping class
  groups, 2020.
\newblock arXiv:2006.09689.

\bibitem{chantraine_cocores_generate}
B.~Chantraine, G.~D. Rizell, P.~Ghiggini, and R.~Golovko.
\newblock Geometric generation of the wrapped {F}ukaya category of {W}einstein
  manifolds and sectors, 2017.
\newblock arXiv:1712.09126.

\bibitem{CE12}
K.~Cieliebak and Y.~Eliashberg.
\newblock {\em From {S}tein to {W}einstein and back}, volume~59 of {\em
  American Mathematical Society Colloquium Publications}.
\newblock American Mathematical Society, Providence, RI, 2012.
\newblock Symplectic geometry of affine complex manifolds.

\bibitem{cornwell_ng_sivek_Lagrangian_concordance}
C.~R. Cornwell, L.~Ng, and S.~Sivek.
\newblock Obstructions to {L}agrangian concordance.
\newblock {\em Algebr. Geom. Topol.}, pages 797--824, 2016.

\bibitem{EGL}
Y.~Eliashberg, S.~Ganatra, and O.~Lazarev.
\newblock Flexible {L}agrangians.
\newblock {\em International Mathematics Research Notices}, 2018.

\bibitem{EM}
Y.~Eliashberg and E.~Murphy.
\newblock {L}agrangian caps.
\newblock {\em Geom. Funct. Anal.}, 23(5):1483--1514, 2013.

\bibitem{FukSS}
K.~Fukaya, P.~Seidel, and I.~Smith.
\newblock The symplectic geometry of cotangent bundles from a categorical
  viewpoint.
\newblock In {\em Homological mirror symmetry}, volume 757 of {\em Lecture
  Notes in Phys.}, pages 1--26. Springer, Berlin, 2009.

\bibitem{Ganatra_thesis}
S.~Ganatra.
\newblock Symplectic cohomology and duality for the wrapped {F}ukaya category,
  2013.

\bibitem{ganatra_generation}
S.~Ganatra, J.~Pardon, and V.~Shende.
\newblock Sectorial descent for wrapped {F}ukaya categories, 2018.
\newblock arXiv:1809.03427.

\bibitem{gromov_hprinciple}
M.~Gromov.
\newblock {\em Partial differential relations}, volume~9 of {\em Ergebnisse der
  Mathematik und ihrer Grenzgebiete (3) [Results in Mathematics and Related
  Areas (3)]}.
\newblock Springer-Verlag, Berlin, 1986.

\bibitem{HKK}
F.~Haiden, L.~Katzarkov, and M.~Kontsevich.
\newblock Flat surfaces and stability structures.
\newblock {\em Publ. Math. Inst. Hautes \'{E}tudes Sci.}, 126:247--318, 2017.

\bibitem{Hutchings_Taubes_chord_conjecture}
M.~Hutchings and C.~H. Taubes.
\newblock Proof of the {A}rnold chord conjecture in three dimensions 1.
\newblock {\em Math. Res. Lett.}, 18(2):295--313, 2011.

\bibitem{kragh_parametrized_ring_spectra_NLC}
T.~Kragh.
\newblock Parametrized ring-spectra and the nearby {L}agrangian conjecture.
\newblock {\em Geom. Topol.}, 17(2):639--731, 2013.

\bibitem{Lazarev_maximal}
O.~Lazarev.
\newblock Maximal contact and symplectic structures, 2018.
\newblock To appear in \textit{Journal of Topology}.

\bibitem{Lazarev_geometric_presentations}
O.~Lazarev.
\newblock Geometric and algebraic presentations for {W}einstein domains, 2019.
\newblock arXiv:1910.01101.

\bibitem{Lazarev_reg_lag}
O.~Lazarev.
\newblock H-principles for regular {L}agrangians.
\newblock {\em J. Symplectic Geom.}, 18(4), 2020.

\bibitem{Lazarev_crit_points}
O.~Lazarev.
\newblock Simplifying {W}einstein {M}orse functions, 2020.
\newblock To appear in Geom. Topol.

\bibitem{Lyubashenko--Manzyuk}
V.~Lyubashenko and O.~Manzyuk.
\newblock Quotients of unital {A}$_\infty$-categories.
\newblock {\em Theory and Applications of Categories}, 20(13):405--496, 2008.

\bibitem{Lyubashenko--Ovsienko}
V.~Lyubashenko and S.~Ovisienko.
\newblock A construction of quotient {A}$_\infty$-categories.
\newblock {\em Homology, Homotopy and Applications}, 8(2):157--203, 2006.

\bibitem{Mohnke_chord}
K.~Mohnke.
\newblock Holomorphic disks and the chord conjecture.
\newblock {\em Ann. of Math. (2)}, 154(1):219--222, 2001.

\bibitem{Murphy11}
E.~Murphy.
\newblock Loose {L}egendrian embeddings in high dimensional contact manifolds,
  2012.
\newblock arXiv:1201.2245.

\bibitem{neeman_chromatic}
A.~Neeman.
\newblock The chromatic tower for d(r).
\newblock {\em Topology}, 31(3):519--532, 1992.

\bibitem{Ritter}
A.~F. Ritter.
\newblock Topological quantum field theory structure on symplectic cohomology.
\newblock {\em J. Topol.}, 6(2):391--489, 2013.

\bibitem{Sylvan_talks}
Z.~Sylvan.
\newblock Talks at {MIT} workshop on {L}efschetz fibrations, 2015.

\bibitem{Sylvan_Orlov_functor}
Z.~Sylvan.
\newblock Orlov and {V}iterbo functors in partially wrapped {F}ukaya
  categories, 2019.
\newblock arXiv:1908.02317.

\bibitem{thomason_thick_subcategories}
R.~Thomason.
\newblock The classification of triangulated subcategories.
\newblock {\em Compositio Math.}, 105(1):1--27, 1997.

\end{thebibliography}

\end{document}